\documentclass[11pt,reqno]{amsart}
\usepackage{hyperref}

\allowdisplaybreaks[4]
\usepackage{verbatim}
\usepackage{cleveref}
\usepackage{amsmath}
\usepackage{amssymb}
\usepackage{mathrsfs}
\usepackage{amsthm}
\usepackage{bm}
\usepackage{bbm}
\usepackage{subcaption}
\usepackage{graphicx}
\usepackage{booktabs}
\usepackage{float}
\usepackage{cases}
\usepackage{geometry}
\usepackage{color}
\theoremstyle{plain}
\newtheorem{lemma}{Lemma}[section]
\newtheorem{thm}[lemma]{Theorem}

\newtheorem{remark}[lemma]{Remark}
\newtheorem{prop}[lemma]{Proposition}
\newtheorem{assumption}{Assumption}

\newcommand{\bbR}{\mathbb R}

\begin{document}
\title[Asymptotic error distributions of symplectic methods]{Symplectic methods for stochastic Hamiltonian systems: asymptotic error distributions and Hamiltonian-specific analysis}
\author{Chuchu Chen, Xinyu Chen, Jialin Hong, Yuqian Miao*}
\address{State Key Laboratory of Mathematical Sciences, Academy of Mathematics and Systems Science, Chinese Academy of Sciences, Beijing 100190, China,
\and 
School of Mathematical Sciences, University of Chinese Academy of Sciences, Beijing 100049, China}
\email{chenchuchu@lsec.cc.ac.cn; chenxinyu@amss.ac.cn; hjl@lsec.cc.ac.cn; miaoyuqian@lsec.cc.ac.cn}
\thanks{This work is funded by the National Key R\&D Program of China under Grant (No. 2024YFA1015900),  by the National Natural Science Foundation of China (No. 12031020, No. 12461160278, and No. 12471386), and by Youth Innovation Promotion Association CAS}
\thanks{*Corresponding author.}
\begin{abstract}
In this paper, we investigate the asymptotic error distributions of symplectic methods for stochastic Hamiltonian systems and further provide Hamiltonian-specific analysis
that clarifies the superiority of symplectic methods.
Our contribution is threefold. First, we derive the asymptotic error distributions of symplectic methods for stochastic Hamiltonian systems with multiplicative noise and additive noise, respectively,
and show that the obtained limiting stochastic processes satisfy equations retaining the Hamiltonian formulations.
Second, we propose a new approach for calculating the asymptotic error distribution, revealing the connection between the stochastic modified equation and the asymptotic error distribution. 
Third, we 
characterize the limiting distribution of the normalized Hamiltonian deviation, thereby illustrating through test equations the superiority of symplectic methods for long-time simulations of the Hamiltonians, even in the limit as the step size tends to zero.
\end{abstract}
\keywords {Stochastic Hamiltonian system $\cdot$ Symplectic methods $\cdot$ Asymptotic error distribution $\cdot$ Stochastic modified equation $\cdot$ Hamiltonian deviation}
\maketitle
\section{Introduction}
Error analysis plays a crucial role in evaluating the accuracy and reliability of stochastic numerical methods, providing insight into how these methods approximate solutions and capture the intrinsic dynamics of stochastic systems.
The asymptotic error distribution, serving as a probabilistic limit theorem, characterizes the distributional behavior of the normalized error process $U^n:=n^p(X^n-X)$ as the discretization parameter $n$ tends to infinity, where $p$ denotes the strong convergence order of the numerical solution $X^n$ to the exact solution $X$. 
The study of the asymptotic error distribution can be traced back to \cite{KurtzProtter1991b}, which establishes the result for the Euler method applied to stochastic differential equations (SDEs) with bounded coefficients. Significant progress has been made on the asymptotic error distribution of explicit Euler-type methods for various stochastic systems (see e.g. \cite{Fukasawa2023,Hong2024spde,HuLiu2016,JacodProtter1998,David2023,Protter2020}). 
In addition, implicit methods with better stability also attract attention in the study of the asymptotic error distribution. 
Authors in \cite{Huyaozhong2023} derive the limiting distribution of $U^n(t)$ for any $t \in [0, T]$ in $\mathbb R^d$ for the backward Euler method applied to SDEs driven by additive fractional Brownian motion, while authors in \cite{Sheng2025} obtain the corresponding result for the $\theta$-method in the context of stochastic Hamiltonian systems with additive noise. Furthermore, this limiting distribution for stochastic Runge--Kutta methods applied to SDEs is given in \cite{jin2025asymptotic}.
Besides, the Crank–Nicolson method for $1$-dimensional SDEs driven by fractional Brownian motion is studied, with the limiting distribution obtained in the Skorohod space $\mathcal D([0,T];\mathbb R)$ (see e.g. \cite{Naganuma2015,Ueda2025}). 

As an important class of stochastic systems, the stochastic Hamiltonian system (SHS) provides a natural stochastic generalization of classical mechanics that reconciles the Hamiltonian structure with the nondifferentiability of Brownian motion and offers applications across chemistry, physics, and engineering. In this paper, we investigate the asymptotic error distributions of symplectic methods for SHSs and further provide Hamiltonian-specific analysis
that clarifies their superiority. This class of numerical methods, which preserves the symplectic structure of the underlying SHS, was pioneered by Milstein et al. (see \cite{Milstein2002, Milstein20022}) and has been further developed over the past decades (see, e.g., \cite{Hongsun2022, Hongwang2019, Milstein2021}). To be specific, we consider the following $2d$-dimensional SHS
\begin{align*} 
    d \begin{pmatrix} P_t \\ Q_t
\end{pmatrix}=\begin{pmatrix}
    0 & I_d \\
    -I_d & 0
\end{pmatrix}DH(P_t,Q_t)dt +\begin{pmatrix}
    0 & I_d \\
    -I_d & 0
\end{pmatrix}D\bar{H}(P_t,Q_t) \circ dW_t
\end{align*}
for $t\in[0,T]$ with initial value $(P_0,Q_0)\in \mathbb{R}^d\times \mathbb{R}^d$. Here, $H$ and $\bar{H}:\mathbb{R}^{2d}\to \mathbb{R}$ are Hamiltonians and $W$ is a 1-dimensional
Brownian motion defined on a complete filtered probability space $(\Omega,\mathcal{F},\{\mathcal{F}_t\}_{t\in[0,T]},\mathbb{P})$.

The first main result of our work is the asymptotic error distribution of a class of symplectic methods applied to SHSs with multiplicative and additive noise, respectively (see \Cref{mainthmmulti} and \Cref{mainthmadd}). To be specific, we obtain the limiting distribution of $U^n$ in the sense of stable convergence in $\mathcal C([0,T]; \mathbb R^{2d})$.
The key step in achieving this is to identify a suitable subspace $\Omega^n \subset \Omega$ such that the continuous numerical solution $(P^n_t,Q^n_t)$ admits an explicit representation on $\Omega^n$, while $\mathbb{P}((\Omega^n)^c)$ can be proved to decay exponentially. 
This construction effectively overcomes the difficulties arising from the non-adaptiveness of the integrands involved in the typical implicit representation, which enables the application of weak limit theorems for solutions to stochastic differential equations, leading to the derivation of the asymptotic error distribution.
Our proposed technique provides an effective framework for obtaining the asymptotic error distribution, which can be adapted to general implicit methods for stochastic differential equations. 
In addition, the intrinsic geometric structure of the SHS is also captured by our results, with the equations governing the obtained asymptotic error distributions retaining Hamiltonian formulations (see \Cref{UisHamiltonian} and \Cref{UisHamiltonianadd}).
As the truncated stochastic modified equation provides a higher-order approximation to the numerical method,
it is natural to ask whether the asymptotic error distribution of a numerical method can be derived via a suitable truncated stochastic modified equation.
We construct such a truncated stochastic modified equation in the strong convergence sense inspired by \cite{Deng2016} and provide a positive answer to this question (see \Cref{modifiedthm} and \Cref{modifiedthmadd}). This constitutes a new approach to deriving the asymptotic error distribution and reveals the connection between the stochastic modified equation and the asymptotic error distribution. Moreover, our results indicate that this approach is more straightforward, as the truncated stochastic modified equation is formulated in terms of integrals with continuous-time adapted integrands, which facilitates the application of weak limit theorems.


Our results also provide a new perspective for explaining the superiority of symplectic methods over non-symplectic ones in simulating SHSs, constituting our third main contribution. Complementing the explanation of this superiority in existing works (see e.g. \cite{Xinyu2024, chen2021asymptotically, chen2023large, Sheng2025, Wanghongsun2016, Wangxinzhang2018}), we show that symplectic methods 
can better simulate the original Hamiltonians over long time intervals, even in the limit as $n$ tends to infinity. Specifically,
we obtain the limiting distribution of the normalized Hamiltonian deviation $n^p(H(P_t^n,Q_t^n)-H(P_t,Q_t))$, based on which we further present explicit expressions for asymptotic distribution of the deviations and their corresponding statistics for the Euler method and symplectic methods applied to the stochastic Kubo oscillator and the linear stochastic oscillator, respectively. Finally, numerical experiments are provided to confirm our theoretical results.

The rest of this paper is organized as follows. Section \ref{preliminaries} provides some preliminaries.
In Sections \ref{Sectionthree} and \ref{Sectionfour}, we establish the asymptotic error distributions of symplectic methods for SHSs with multiplicative noise and additive noise, respectively, upon which we show the Hamiltonian-specific results. 
In Section \ref{Sectionfive}, we construct the stochastic modified equation in the strong convergence sense and give the new approach for obtaining the asymptotic error distribution. 
Numerical experiments are presented in \Cref{Sectionsix}, demonstrating the superiority of the symplectic methods over the Euler method in simulating the Hamiltonians.

At the end of this section, we give some notations for the following content. 
Let 
$\mathcal D([0,T];{\mathbb{R}^d})$ denote the space of the c\`adl\`ag functions from 
$[0,T]$ to $\mathbb{R}^d$, equipped with the Skorohod topology, which is a Polish space. Let 
$\mathcal C([0,T];{\mathbb{R}^d})$ denote the space of continuous functions, which is a subspace of $\mathcal D([0,T];{\mathbb{R}^d})$.
The Skorohod distance between two continuous functions $x,y \in 
\mathcal C([0,T];{\mathbb{R}^d})$ reduces to the uniform distance, i.e., 
$\sup_{t\in [0,T]} \|x(t)-y(t)\|$, where $\|\cdot\|$ denotes the Euclidean norm. 
Denote by $\mathcal C(\mathbb R^d; \mathbb R^m)$ (resp. $\mathcal C^k(\mathbb R^d; \mathbb R^m)$) the space of continuous (resp. $k$th continuously differentiable) functions from $\mathbb R^d$ to $\mathbb R^m$. 
Throughout and without ambiguity, we denote by $\Rightarrow$ (resp. $\Rightarrow^{stably}$, $\to^{\mathbb{P}}$) the convergence in distribution (resp. the stable convergence in distribution, the convergence in probability), by $D$ the derivative operator, and by $C$ an arbitrary constant whose value may vary from one place to another.  
For $a\in \bbR$, let $[a]$ denote the maximal integer smaller than or equal to $a$. 

\section{Preliminaries} \label{preliminaries}



This section presents some weak limit theorems. We start with  
some basic results for the weak convergence of stochastic processes. 
Let $X_n = (X_n^1, \dots, X_n^d)$ and $X = (X^1, \dots, X^d)$ be $\mathcal{D}([0,T];\mathbb{R}^d)$-valued random variables. It is important to note that the component-wise convergence $X_n^i \Rightarrow X^i$ in $\mathcal{D}([0,T];\mathbb{R})$ for each $i=1,\dots,d$ does not imply the joint convergence $(X_n^1, \dots, X_n^d) \Rightarrow (X^1, \dots, X^d)$ in $\mathcal{D}([0,T];\mathbb{R}^d)$. A sufficient condition to ensure the desired joint convergence is the continuity of each component $X^i$ of the limiting process (see \cite[Proposition 5]{Naganuma2015}).
Moreover, if $F:\mathbb{R}^d\to \mathbb{R}^{d\times m}$ is continuous,
then the convergence $(X_n,Y_n)\Rightarrow(X,Y)$ in $\mathcal{D}([0,T];\mathbb{R}^d\times\mathbb{R}^m)$ further implies $(F(X_n),Y_n)\Rightarrow (F(X),Y)$ in $\mathcal{D}([0,T];\mathbb{R}^{d\times m}\times\mathbb{R}^m)$, which is a special case of the result in \cite{KurtzProtter1991a}.

In the following proposition, we present a criterion for the convergence of sequences of solutions to stochastic differential equations, 
which serves as
a critical tool in our analysis.
\begin{prop} \label{propmain}
For each $n\in\mathbb{N}^+$, let $U_n$ be $\{\mathcal{F}_t\}$-adapted processes with sample paths in $\mathcal{D}([0,T];\mathbb{R}^d)$ and $Y$ be an $\mathbb{R}^m$-valued $\{\mathcal{F}_t\}$-semimartingale. Suppose that $(U_n,Y)\Rightarrow (U,Y)$ 
in $\mathcal{D}([0,T];\mathbb{R}^d\times\mathbb{R}^m)$. Let $\eta_n$ be a right continuous, nondecreasing $\{\mathcal{F}_t\}$-adapted process and assume that $\eta_n(t)\le t$ and $\eta_n(t)\to t$ for all $t\ge0$, $F:\mathbb{R}^d \to \mathbb{R}^{d\times m}$ be globally Lipschitz continuous, and $X_n$ satisfy
\begin{align*}
    X_n(t) = U_n(t) + \int_0^t F(X_n\circ\eta_n(s-))dY(s).
\end{align*}
Then $\{(X_n,U_n,Y)\}$ is relatively compact and any limit point $(X,U,Y)$ satisfies
\begin{align} \label{equX}
    X(t) = U(t) + \int_0^t F(X(s-))dY(s).
\end{align}
If there exists a unique (strong) solution $X$ of \eqref{equX}, then $(X_n,U_n,Y) \Rightarrow (X,U,Y)$ in $\mathcal{D}([0,T];$ $\mathbb{R}^d\times\mathbb{R}^d\times\mathbb{R}^m)$. Moreover, if $(U_n,Y)\to^{\mathbb{P}} (U,Y)$, then we have $(X_n,U_n,Y) \to^{\mathbb{P}} (X,U,Y)$.
\end{prop}
\begin{proof}
     Let $Z_n(t):=\int_0^t F(X_n\circ\eta_n(s-))dY(s)$.It follows from \cite[Theorem 3.1]{JacodProtter1998} that $\{(Z_n,Y)\}$ is relatively compact. Since the stochastic integral $Z_n$ has a discontinuity only when $Y$ has a discontinuity, and $\{(U_n,Y)\}$ is relatively compact, it is clear that $\{(X_n,U_n,Y)\}$ is relatively compact. Since $F$ is continuous, we can obtain that $\{(X_n,F(X_n),U_n,Y)\}$ is relatively compact. The fact that any limit point satisfies \eqref{equX} then follows from \cite[Lemma 3.2]{KurtzProtter1991b}. Using the uniqueness assumption, we have $(X_n,U_n,Y) \Rightarrow (X,U,Y)$.
\end{proof}
\begin{remark} \label{rmk1}
    In fact, if $X_n$ satisfies
\begin{align*}
    X_n(t) = U_n(t) + \int_0^t F(X_n(s-))dY(s),
\end{align*}
in \Cref{propmain} ($F$ is bounded and continuous or $F$ is globally Lipschitz), the result $(X_n,U_n,Y) \Rightarrow (X,U,Y)$ still holds.
\end{remark}


Finally, we provide the definition of stable convergence and introduce a criterion for it. 
Let $\{X_n\}$ be a sequence of random variables defined on the same probability space 
$(\Omega, \mathcal{F}, \mathbb{P})$, taking values in a Polish space $E$. 
We say that $X_n$ converges stably to $X$, 
if $\lim_{n \to \infty} \mathbb{E}[U f(X_n)] = \tilde{\mathbb{E}}[U f(X)]$
for every bounded continuous function $f:E \to \mathbb{R}$ and all bounded measurable 
$\mathbb{R}$-valued random variables $U$, where $\tilde{\mathbb{E}}$ denotes the expectation 
on an extension of the original probability space.
In fact, this convergence is stronger than the convergence in distribution. 

According to \cite[Section 2]{Jacod1997} and \cite[Lemma 2.1]{JacodProtter1998}, we give the following lemma.
\begin{lemma} \label{lemmastable}
    Let $Y_n$ be defined on $(\Omega,\mathcal{F},\mathbb{P})$ with values in another metric space $E'$.\\
    (i) Let $Y$ be defined on $(\Omega,\mathcal{F},\mathbb{P})$. If $Y_n\to^{\mathbb{P}} Y$
    and $X_n\Rightarrow^{stably} X$, then we have $(X_n,Y_n)\Rightarrow^{stably}(X,Y)$ for the product topology $E\times E'$.\\
    (ii) Conversely, if $(X_n,Y_n)\Rightarrow (X,Y)$ and $Y$ generates the $\sigma$-field $\mathcal{F}$, we can realize that $X$ is defined on an extension of $(\Omega,\mathcal{F},\mathbb{P})$ and $X_n\Rightarrow^{stably} X$.
\end{lemma}

\section{Asymptotic error distributions of symplectic methods for SHS with multiplicative noise} \label{Sectionthree}
In this section, we investigate the asymptotic error distributions of symplectic methods for SHS with multiplicative noise. 
We start with the asymptotic error distribution of the symplectic Euler method for the 2-dimensional SHS with multiplicative noise (see \Cref{mainthm}) and then extend the result to a class of symplectic methods for the $2d$-dimensional SHS 
(see \Cref{mainthmmulti}).
Furthermore, we show the Hamiltonian-specific results based on the obtained asymptotic error distributions (see \Cref{UisHamiltonian} and \Cref{thmdeltaH}). 

\subsection{Symplectic Euler method for the case of multiplicative noise} \label{3.1}
We first consider a 2-dimensional SHS:
\begin{align} \label{Generalequ}
    d \begin{pmatrix} P_t \\ Q_t
\end{pmatrix}=\begin{pmatrix}
    f(P_t,Q_t) \\
    g(P_t,Q_t)
\end{pmatrix}dt +\begin{pmatrix}
    a(P_t,Q_t) \\ b(P_t,Q_t)
\end{pmatrix} \circ dW_t, \quad t\in (0,T],
\end{align}
with initial value $(P_0,Q_0)\in \mathbb{R}^2$, where $W$ is a 1-dimensional Brownian motion on the filtered complete probability space $(\Omega,\mathcal{F},\{\mathcal{F}_t\}_{t\ge 0},\mathbb{P})$, $f:=-\frac{\partial H}{\partial Q}$, $g:=\frac{\partial H}{\partial P}$, $a:=-\frac{\partial \bar{H}}{\partial Q}$, $b:=\frac{\partial \bar{H}}{\partial P}$, and $H,\bar{H}:\mathbb{R}^2\to \mathbb{R}$ are Hamiltonian functions.
Using the Stratonovich--It\^{o} conversion formula, the solution of \eqref{Generalequ} satisfies
\begin{align}
    P_t =&\  P_0+\int_0^t (f+\frac{1}{2}a_1'a+\frac{1}{2}a_2'b)(P_s,Q_s)ds + \int_0^t a(P_s,Q_s)dW_s, \label{p} \\
    Q_t =& \ Q_0+\int_0^t (g+\frac{1}{2}b_1'a+\frac{1}{2}b_2'b)(P_s,Q_s)ds + \int_0^t b(P_s,Q_s)dW_s, \label{q}
\end{align}
where $a_1':=\frac{\partial a}{\partial P},a_2':=\frac{\partial a}{\partial Q}$, $b_1':=\frac{\partial b}{\partial P}$, and $b_2':=\frac{\partial b}{\partial Q}$.

Without loss of generality, we assume that $T$ is an integer. We consider a uniform partition of the interval $[0,T]$ with step size $\frac{1}{n}$: $0=t_0<t_1<\cdots<t_{nT}=T$, where $t_k:=\frac{k}{n}$ for $n\in\mathbb N^+$ and $k=0,1,\dots,nT$. Then the symplectic Euler method is defined by $(P_0^n,Q_0^n):=(P_0,Q_0)$ and
\begin{align*} 
    \begin{cases}
    P_{t_{k+1}}^n = P_{t_k}^n + (f+\frac{1}{2}a_2'b-\frac{1}{2}a_1'a)(P_{t_{k+1}}^n,Q_{t_k}^n)\frac 1n+a(P_{t_{k+1}}^n,Q_{t_k}^n)\Delta \hat{W}_k, \\
    Q_{t_{k+1}}^n = Q_{t_k}^n + (g+\frac{1}{2}b_2'b-\frac{1}{2}b_1'a)(P_{t_{k+1}}^n,Q_{t_k}^n)\frac 1n+b(P_{t_{k+1}}^n,Q_{t_k}^n)\Delta \hat{W}_k,
    \end{cases}
\end{align*}
for $k=0,\dots,nT-1$. Here, we use the truncated random variable $\Delta \hat{W}_k:=\sqrt{1/n}\,\zeta_{k}$ instead of $\Delta W_k=W_{t_{k+1}}-W_{t_k}=\sqrt{1/n}\,\xi_k$,
where $\xi_0,\xi_1,\dots,\xi_{nT}$ are independent $\mathcal N(0,1)$-distributed random variables, and $\zeta_k$ is defined by
\begin{align*}
    \zeta_{k}:=
    \begin{cases}
        \xi_k, \quad &|\xi_k|\le A_n,\\
        A_n, \quad & \xi_k>A_n,\\
        -A_n, \quad & \xi_k<-A_n,
    \end{cases}
\end{align*}
with $A_n:=\sqrt{2\rho \ln n}$ and the constant $\rho> 2$. The following properties hold:
\begin{align}
    &\mathbb{E}[|\xi_k-\zeta_{k}|^p] \le C(p)\big(\frac{1}{n}\big)^\rho,\quad \forall p\ge1, \label{truncated1} \\
    &\big( \mathbb{E}[| \xi_k^\beta-\zeta_{k}^\beta|^p]\big)^{\frac 1p} \le C(p,\beta,\epsilon,\rho)\big(\frac{1}{n}\big)^{\rho/p-\epsilon}, \quad \forall \beta\in \mathbb{N}^+, \ p\ge 1,\ \epsilon\in(0,\rho/p).
    \label{truncated2}
\end{align}
The proof of \eqref{truncated1} is standard similar to  \cite[Eq.~(1.3.28)]{Milstein2021}. 
For \eqref{truncated2},
we have $\mathbb{E}[|\xi_k^\beta-\zeta_{k}^\beta|^p]=\mathbb{E}[|\xi_k-\zeta_{k}|^p|\xi_k^{\beta-1}+\xi_k^{\beta-2}\zeta_{k}+\dots+\zeta_{k}^{\beta-1}|^p]\le C(p,\beta,\epsilon,\rho)(\mathbb{E}|\xi_k-\zeta_{k}|^{\frac{p}{1-\epsilon p/\rho}})^{1-\epsilon p/\rho}\le C(p,\beta,\epsilon,\rho)(1/n)^{\rho-p\epsilon}$ for any $\epsilon\in(0,\rho/p)$.


Since the symplectic Euler method is implicit, we need to ensure its well-definedness. For any fixed $\epsilon\in(0,\frac{1}{2})$, we introduce a subset $\Omega^n \subset \Omega$:
\begin{align*}
    \Omega^n = \Bigg\{\omega \in \Omega: \sup_{\substack{0\le s,t\le T \\|t-s|\le \frac{1}{n}}} |W_t(\omega)-W_s(\omega)| \le \big(\frac{1}{n}\big)^{\frac{1}{2}-\epsilon} \Bigg\}.
\end{align*}
In the following proposition, we show that $\mathbb{P}((\Omega^n)^c)$ decays exponentially.
\begin{prop} \label{propOmeganc}
    There exist constants $C_1,C_2>0$ such that $\mathbb{P}((\Omega^n)^c)\le C_2n^{1-\epsilon}e^{-C_1n^{2\epsilon}}$.
\end{prop}
\begin{proof}
    For $t_i=i/n,i=0,1,\dots,nT$, by the triangle inequality, there exists a constant $C>0$ such that
    \begin{align*}
        \sup_{\substack{0\le s,t\le T \\|t-s|\le \frac{1}{n}}} |W_t-W_s| \le C \sup_{i\in\{0,\ldots,nT\}} \sup_{0\le \delta\le \frac{1}{n}} |W_{t_i+\delta}-W_{t_i}|.
    \end{align*}
    Let $Z_i=\sup_{0\le \delta\le \frac{1}{n}} |W_{t_i+\delta}-W_{t_i}|$, then 
    \begin{align*}
        \mathbb{P}((\Omega^n)^c)\le\mathbb{P}\Big(\sup_{i\in\{0,\ldots,nT\}}  Z_i > C\big(\frac{1}{n}\big)^{\frac{1}{2}-\epsilon}\Big) \le nT\,\mathbb{P}\Big(Z_0>C\big(\frac{1}{n}\big)^{\frac{1}{2}-\epsilon}\Big).
    \end{align*}
    Thus, there exists a constant $C_1>0$ such that
    \begin{align*}
        \mathbb{P}\Big(Z_0>C\big(\frac{1}{n}\big)^{\frac{1}{2}-\epsilon}\Big)=\mathbb{P}\Big(\sup_{0\le\delta\le \frac{1}{n}}|W_\delta|> C\big(\frac{1}{n}\big)^{\frac{1}{2}-\epsilon}  \Big)\le Cn^{-\epsilon}e^{-C_1n^{2\epsilon}},
    \end{align*}
    where the last inequality is obtained by the reflection principle for Brownian motion (see \cite[Theorem 21.19]{Klenke2008}). Hence, we conclude that there exist constants $C_1,C_2>0$ such that $\mathbb{P}((\Omega^n)^c)\le C_2n^{1-\epsilon}e^{-C_1n^{2\epsilon}}$. 
\end{proof}

On $\Omega^n$, we define the continuous version $(P_t^n,Q_t^n)$ by
\begin{align} \label{SymEuler}
    \begin{cases}
    P_t^n = P_{t_k}^n + (f+\frac{1}{2}a_2'b-\frac{1}{2}a_1'a)(P_t^n,Q_{t_k}^n)(t-t_k)+a(P_t^n,Q_{t_k}^n)\Delta\hat{W}_{k,t}, \\
    Q_t^n = Q_{t_k}^n + (g+\frac{1}{2}b_2'b-\frac{1}{2}b_1'a)(P_t^n,Q_{t_k}^n)(t-t_k)+b(P_t^n,Q_{t_k}^n)\Delta\hat{W}_{k,t},
    \end{cases}
    \quad t\in (t_k,t_{k+1}],
\end{align}
for $k=0,\dots,nT-1$, where $\Delta\hat{W}_{k,t}:=\sqrt{t-t_k}\zeta_{k}$. And let $(P_t^n,Q_t^n) = (P_0,Q_0)$, $t\in[0,T]$ on $(\Omega^n)^c$. To simplify the notation, 
we denote $u:=f+\frac{1}{2}a_2'b-\frac{1}{2}a_1'a$, $v:=g+\frac{1}{2}b_2'b-\frac{1}{2}b_1'a$, and $u_1':=\frac{\partial u}{\partial P},u_2':=\frac{\partial u}{\partial Q},v_1':=\frac{\partial v}{\partial P},v_2':=\frac{\partial v}{\partial Q}$.

\begin{assumption} \label{assumption1}
    (\romannumeral 1) $f,g,a,b,a_1'a,a_2'b,b_1'a,b_2'b:\mathbb{R}^2\to \mathbb{R}$ are globally Lipschitz continuous.\\
    (\romannumeral 2) Let $f,g\in C^3(\mathbb{R}^2;\mathbb{R})$, $a,b\in C^4(\mathbb{R}^2;\mathbb{R})$ and assume that the $k$th order derivatives $D^k f,D^k g$ for $k= 2,3$ and $D^k a,D^k b$ for $k= 2,3,4$ are functions growing at most polynomially.
\end{assumption}

We then show that, under \Cref{assumption1}, the continuous numerical solution is indeed well-defined and admits an explicit expression on $\Omega^n$ for large $n$.

\begin{prop} \label{mainprop}
Let $\epsilon_1=\frac{1}{4\sup|u_1'|}$ ($\epsilon_1=\infty$ if $\sup|u_1'|=0$) and $\epsilon_2=\frac{1}{4\sup|a_1'|}$ ($\epsilon_2=\infty$ if $\sup|a_1'|=0$). Suppose 
\begin{align} \label{n}
    n > \max\big\{ 4\sup|u_1'|, (4\sup|a_1'|)^{1/(\frac{1}{2}-\epsilon)}
    \big\},
\end{align}
then, under Assumption \ref{assumption1}, $(P_t^n,Q_t^n)$ is well-defined and 
\begin{align}
    P_t^n =&\  P_{t_k}^n+u(P_{t_k}^n,Q_{t_k}^n)(t-t_k)+a(P_{t_k}^n,Q_{t_k}^n)\Delta \hat{W}_{k,t}+(a_1'a)(P_{t_k}^n,Q_{t_k}^n)\Delta \hat{W}_{k,t}^2 \nonumber\\
    &+ (u_1'u)(P_{t_k}^n,Q_{t_k}^n)(t-t_k)^2+(u_1'a+a_1'u)(P_{t_k}^n,Q_{t_k}^n)(t-t_k)\Delta \hat{W}_{k,t}\nonumber \\
    &+R_1(P_{t_k}^n,Q_{t_k}^n,t-t_k,\Delta \hat{W}_{k,t}), \quad t\in (t_k,t_{k+1}], \label{explicitp} \\
    Q_t^n =& \ Q_{t_k}^n+v(P_{t_k}^n,Q_{t_k}^n)(t-t_k)+b(P_{t_k}^n,Q_{t_k}^n)\Delta \hat{W}_{k,t}+(b_1'a)(P_{t_k}^n,Q_{t_k}^n)\Delta \hat{W}_{k,t}^2 \nonumber\\
    &+ (v_1'u)(P_{t_k}^n,Q_{t_k}^n)(t-t_k)^2+(v_1'a+b_1'u)(P_{t_k}^n,Q_{t_k}^n)(t-t_k)\Delta \hat{W}_{k,t}\nonumber \\
    &+R_2(P_{t_k}^n,Q_{t_k}^n,t-t_k,\Delta \hat{W}_{k,t}), \quad t\in (t_k,t_{k+1}]
    \label{explicitq}
\end{align}
on $\Omega^n$ for $k=0,1,\ldots,nT-1$, where $R_1$ and $R_2$ are real-valued continuous functions on $\mathbb{R}^2\times (-\epsilon_1,\epsilon_1)\times(-\epsilon_2,\epsilon_2)$ 
with expansion
$$R_1(\zeta,h_1,h_2)=\sum_{\alpha+\beta=3,\alpha,\beta\in\mathbb{N}}R_1^{\alpha,\beta}(\zeta,\lambda_1h_1,\lambda_2h_2)h_1^\alpha h_2^\beta,$$ $$R_2(\zeta,h_1,h_2)=\sum_{\alpha+\beta=3,\alpha,\beta\in\mathbb{N}}R_2^{\alpha,\beta}(\zeta,\bar{\lambda}_1h_1,\bar{\lambda}_2h_2)h_1^\alpha h_2^\beta,$$
for some constants $\lambda_1,\lambda_2,\bar{\lambda}_1,\bar{\lambda}_2\in[0,1]$, and functions $R_1^{\alpha,\beta},R_2^{\alpha,\beta}$with $\alpha+\beta=3,\alpha,\beta\in \mathbb{N}$ growing at most polynomially.
\end{prop}
\begin{proof}   
For $\omega\in\Omega^n$ and $t\in(t_k,t_{k+1}]$, by \eqref{n}  we have
 $|t-t_k|<\frac{1}{4\sup|u_1'|}$ and $|W_t-W_{t_k}|<\frac{1}{4\sup|a_1'|}$. 
 So that $|\Delta \hat{W}_{k,t}|<\frac{1}{4\sup|a_1'|}$. For a fixed $\xi^0=(\xi_1^0,\xi_2^0)^{\top}\in \mathbb{R}^2$, we define
\begin{align*}
    F(\xi_1,\xi_2,h_1,h_2)=\begin{pmatrix}
        \xi_1-\xi_1^0-u(\xi_1,\xi_2^0)h_1-a(\xi_1,\xi_2^0)h_2 \\
        \xi_2-\xi_2^0-v(\xi_1,\xi_2^0)h_1-b(\xi_1,\xi_2^0)h_2
    \end{pmatrix}.
\end{align*}
Using the fact that 
\begin{align*}
    \frac{\partial F}{\partial(\xi_1,\xi_2)} = \begin{pmatrix}
        1-u_1'(\xi_1,\xi_2^0)h_1-a_1'(\xi_1,\xi_2^0)h_2 & 0 \\
        -v_1'(\xi_1,\xi_2^0)h_1-b_1'(\xi_1,\xi_2^0)h_2 & 1
    \end{pmatrix},
\end{align*}
we obtain $|\frac{\partial F}{\partial(\xi_1,\xi_2)}|=1-u_1'(\xi_1,\xi_2^0)h_1-a_1'(\xi_1,\xi_2^0)h_2>\frac{1}{2}$ for $|h_1|<\epsilon_1,|h_2|<\epsilon_2$. By the implicit function theorem, there exist functions $l_1$ and $l_2$ such that $F(l_1(h_1,h_2),l_2(h_1,h_2),h_1,h_2)=0$ with $l_1(0,0)=\xi_1^0$ and $l_2(0,0)=\xi_2^0$. The Taylor formula yields
\begin{align*}
    \xi_1 = l_1(h_1,h_2) =& \,\xi_1^0+u(\xi^0)h_1+a(\xi^0)h_2+(a_1'a)(\xi^0)h_2^2 \\
    &+(u_1'u)(\xi^0)h_1^2+(a_1'u+u_1'a)(\xi^0)h_1h_2 +R_1(\xi^0,h_1,h_2), \\
    \xi_2 = l_2(h_1,h_2) =&\, \xi_2^0+v(\xi^0)h_1+b(\xi^0)h_2+(b_1'a)(\xi^0)h_2^2 \\
    &+(v_1'u)(\xi^0)h_1^2+(v_1'a+b_1'u)(\xi^0)h_1h_2+R_2(\xi^0,h_1,h_2),
\end{align*}
where
\begin{align*}
    &R_1(\xi^0,h_1,h_2)=\sum_{\alpha+\beta=3,\alpha,\beta\in\mathbb{N}}R_1^{\alpha,\beta}(\xi^0,\lambda_1h_1,\lambda_2h_2)h_1^\alpha h_2^\beta,\\
    &R_2(\xi^0,h_1,h_2)=\sum_{\alpha+\beta=3,\alpha,\beta\in\mathbb{N}}R_2^{\alpha,\beta}(\xi^0,\bar{\lambda}_1h_1,\bar{\lambda}_2h_2)h_1^\alpha h_2^\beta,
\end{align*}
with $\lambda_1,\lambda_2,\bar{\lambda}_1,\bar{\lambda}_2\in[0,1]$, and $R_1^{\alpha,\beta},R_2^{\alpha,\beta}$ with $\alpha+\beta=3,\alpha,\beta\in \mathbb{N}$ are real-valued continuous functions on $\mathbb{R}^2\times (-\epsilon_1,\epsilon_1)\times(-\epsilon_2,\epsilon_2)$. Under Assumption \ref{assumption1}, we know that for each $\alpha,\beta$, there exist $K_1^{\alpha,\beta},K_2^{\alpha,\beta}>0$ and $\gamma_1^{\alpha,\beta},\gamma_2^{\alpha,\beta}>0$ such that
\begin{align*}
    &\Big|R_1^{\alpha,\beta}(\xi^0,\lambda_1h_1,\lambda_2h_2)\Big|\le K_1^{\alpha,\beta}(1+\|\xi^0\|^{\gamma_1^{\alpha,\beta}}),\\
    &\Big|R_2^{\alpha,\beta}(\xi^0,\lambda_1h_1,\lambda_2h_2)\Big|\le K_2^{\alpha,\beta}(1+\|\xi^0\|^{\gamma_2^{\alpha,\beta}}).
\end{align*}
We conclude that $(P_t^n,Q_t^n)$ is well-defined and satisfies \eqref{explicitp} and \eqref{explicitq} on $\Omega^n$, which ends the proof.
\end{proof} 

Based on \Cref{mainprop}, we derive the integral form of $(P_t^n,Q_t^n)$.
Let $\eta_n(t)=\frac{[nt]}{n}$, it follows from \Cref{mainprop} that
\begin{align*}
    P_t^n =& \ P_0^n+\sum_{k=0}^{[nt]-1}u(P_{t_k}^n,Q_{t_k}^n)(t_{k+1}-t_k)+u(P_{\eta_n(t)}^n,Q_{\eta_n(t)}^n)(t-\eta_n(t)) \\
    &+\sum_{k=0}^{[nt]-1}a(P_{t_k}^n,Q_{t_k}^n)(W_{t_{k+1}}-W_{t_k})+a(P_{\eta_n(t)}^n,Q_{\eta_n(t)}^n)(W_t-W_{\eta_n(t)}) \\
    &+ \sum_{k=0}^{[nt]-1} (a_1'a)(P_{t_k}^n,Q_{t_k}^n)(W_{t_{k+1}}-W_{t_k})^2+(a_1'a)(P_{\eta_n(t)}^n,Q_{\eta_n(t)}^n)(W_t-W_{\eta_n(t)})^2+R_{P,2}^n(t)+R_{P,3}^n(t)
\end{align*}
on $\Omega^n$, where
\begin{align*}
    R_{P,2}^n(t):=&\sum_{k=0}^{[nt]-1}(u_1'u)(P_{t_k}^n,Q_{t_k}^n)(t_{k+1}-t_k)^2+(u_1'u)(P_{\eta_n(t)}^n,Q_{\eta_n(t)}^n)(t-\eta_n(t))^2 \\
    &+\sum_{k=0}^{[nt]-1}(u_1'a+a_1'u)(P_{t_k}^n,Q_{t_k}^n)(t_{k+1}-t_k)(W_{t_{k+1}}-W_{t_k}) \\
    &+(u_1'a+a_1'u)(P_{\eta_n(t)}^n,Q_{\eta_n(t)}^n)(t-\eta_n(t))(W_t-W_{\eta_n(t)}) \\
    &+\sum_{k=0}^{[nt]-1}\sum_{\alpha+\beta=3,\alpha,\beta\in\mathbb{N}}R_1^{\alpha,\beta}(P_{t_k}^n,Q_{t_k}^n,\lambda_1(t_{k+1}-t_k),\lambda_2\Delta\hat{W}_k)(t_{k+1}-t_k)^\alpha(W_{t_{k+1}}-W_{t_k})^\beta\\
    &+\sum_{\alpha+\beta=3,\alpha,\beta\in\mathbb{N}}R_1^{\alpha,\beta}(P_{\eta_n(t)}^n,Q_{\eta_n(t)}^n,\lambda_1(t-\eta_n(t)),\lambda_2\Delta\hat{W}_{[nt],t})(t-\eta_n(t))^\alpha(W_t-W_{\eta_n(t)})^\beta, 
\end{align*}
and $R_{P,3}^n$ denotes the error arising from the truncated increments of $W$, that is,
\begin{align*}
    R_{P,3}^n(t):=&\sum_{k=0}^{[nt]-1}a(P_{t_k}^n,Q_{t_k}^n)\sqrt{t_{k+1}-t_k}(\zeta_{k}-\xi_k)+a(P_{\eta_n(t)}^n,Q_{\eta_n(t)}^n)\sqrt{t-\eta_n(t)}(\zeta_{[nt]}-\xi_{[nt]})\\
    &+\sum_{k=0}^{[nt]-1} (a_1'a)(P_{t_k}^n,Q_{t_k}^n)(t_{k+1}-t_k)(\zeta_{k}^2-\xi_k^2)+(a_1'a)(P_{\eta_n(t)}^n,Q_{\eta_n(t)}^n)(t-\eta_n(t))(\zeta_{[nt]}^2-\xi_{[nt]}^2)\\
    &+\sum_{k=0}^{[nt]-1}(u_1'a+a_1'u)(P_{t_k}^n,Q_{t_k}^n)(t_{k+1}-t_k)^{3/2}(\zeta_{k}-\xi_k) \\
    &+(u_1'a+a_1'u)(P_{\eta_n(t)}^n,Q_{\eta_n(t)}^n)(t-\eta_n(t))^{3/2}(\zeta_{[nt]}-\xi_{[nt]})\\
    &+\sum_{k=0}^{[nt]-1}\sum_{\alpha+\beta=3,\alpha,\beta\in\mathbb{N}}R_1^{\alpha,\beta}(P_{t_k}^n,Q_{t_k}^n,\lambda_1(t_{k+1}-t_k),\lambda_2\Delta\hat{W}_k)(t_{k+1}-t_k)^{\alpha+\beta/2}(\zeta_{k}^\beta-\xi_k^\beta)\\
    &+\sum_{\alpha+\beta=3,\alpha,\beta\in\mathbb{N}}R_1^{\alpha,\beta}(P_{\eta_n(t)}^n,Q_{\eta_n(t)}^n,\lambda_1(t-\eta_n(t)),\lambda_2\Delta\hat{W}_{[nt],t})(t-\eta_n(t))^{\alpha+\beta/2}(\zeta_{[nt]}^\beta-\xi_{[nt]}^\beta).
\end{align*}
Combining with
\begin{align*}
    &(W_t-W_s)^2 = t-s+2\int_s^t (W_{\tau}-W_s)dW_{\tau},
\end{align*}
the explicit expression of $P_t^n$ in \eqref{explicitp} leads to
\begin{equation} \label{pn}
    \begin{split}
        P_t^n =&\  P_0 + \int_0^t (u+a_1'a)(P_{\eta_n(s)}^n,Q_{\eta_n(s)}^n)ds + \int_0^ta(P_{\eta_n(s)}^n,Q_{\eta_n(s)}^n)dW_s + R_P^n(t),
    \end{split}
\end{equation}
on $\Omega^n$, where $R_P^n:=R_{P,1}^n+R_{P,2}^n+R_{P,3}^n$ and 
\begin{align*}
    R_{P,1}^n(t):= 2\int_0^t (a_1'a)(P_{\eta_n(s)}^n,Q_{\eta_n(s)}^n)(W_s-W_{\eta_n(s)})dW_s.
\end{align*}
By applying an analogous argument to $Q_t^n$ in \eqref{explicitq}, we obtain 
\begin{equation} \label{qn}
    \begin{split}
        Q_t^n =&\  Q_0 + \int_0^t (v+b_1'a)(P_{\eta_n(s)}^n,Q_{\eta_n(s)}^n)ds + \int_0^tb(P_{\eta_n(s)}^n,Q_{\eta_n(s)}^n)dW_s +R_Q^n(t),
    \end{split}
\end{equation}
on $\Omega^n$, where $R_Q^n:=R_{Q,1}^n+R_{Q,2}^n+R_{Q,3}^n$, 
\begin{align*}
    R_{Q,1}^n(t):= 2\int_0^t (b_1'a)(P_{\eta_n(s)}^n,Q_{\eta_n(s)}^n)(W_s-W_{\eta_n(s)})dW_s,
\end{align*}
and $R_{Q,2}^n,R_{Q,3}^n$ have similar formulations to $R_{P,2}^n,R_{P,3}^n$, respectively, differing only in the coefficients. To be specific,
\begin{align*}
    R_{Q,2}^n(t):=&\sum_{k=0}^{[nt]-1}(v_1'u)(P_{t_k}^n,Q_{t_k}^n)(t_{k+1}-t_k)^2+(v_1'u)(P_{\eta_n(t)}^n,Q_{\eta_n(t)}^n)(t-\eta_n(t))^2 \\
    &+\sum_{k=0}^{[nt]-1}(v_1'a+b_1'u)(P_{t_k}^n,Q_{t_k}^n)(t_{k+1}-t_k)(W_{t_{k+1}}-W_{t_k}) \\
    &+(v_1'a+b_1'u)(P_{\eta_n(t)}^n,Q_{\eta_n(t)}^n)(t-\eta_n(t))(W_t-W_{\eta_n(t)}) \\
    &+\sum_{k=0}^{[nt]-1}\sum_{\alpha+\beta=3,\alpha,\beta\in\mathbb{N}}R_2^{\alpha,\beta}(P_{t_k}^n,Q_{t_k}^n,\bar{\lambda}_1(t_{k+1}-t_k),\bar{\lambda}_2\Delta\hat{W}_k)(t_{k+1}-t_k)^\alpha(W_{t_{k+1}}-W_{t_k})^\beta\\
    & +\sum_{\alpha+\beta=3,\alpha,\beta\in\mathbb{N}}R_2^{\alpha,\beta}(P_{\eta_n(t)}^n,Q_{\eta_n(t)}^n,\bar{\lambda}_1(t-\eta_n(t)),\bar{\lambda}_2\Delta\hat{W}_{[nt],t})(t-\eta_n(t))^\alpha(W_t-W_{\eta_n(t)})^\beta,\\
    R_{Q,3}^n(t):=&\sum_{k=0}^{[nt]-1}b(P_{t_k}^n,Q_{t_k}^n)\sqrt{t_{k+1}-t_k}(\zeta_{k}-\xi_k)+b(P_{\eta_n(t)}^n,Q_{\eta_n(t)}^n)\sqrt{t-\eta_n(t)}(\zeta_{[nt]}-\xi_{[nt]})\\
    &+\sum_{k=0}^{[nt]-1} (b_1'a)(P_{t_k}^n,Q_{t_k}^n)(t_{k+1}-t_k)(\zeta_{k}^2-\xi_k^2)+(b_1'a)(P_{\eta_n(t)}^n,Q_{\eta_n(t)}^n)(t-\eta_n(t))(\zeta_{[nt]}^2-\xi_{[nt]}^2)\\
    &+\sum_{k=0}^{[nt]-1}(v_1'a+b_1'u)(P_{t_k}^n,Q_{t_k}^n)(t_{k+1}-t_k)^{3/2}(\zeta_{k}-\xi_k) \\
    &+(v_1'a+b_1'u)(P_{\eta_n(t)}^n,Q_{\eta_n(t)}^n)(t-\eta_n(t))^{3/2}(\zeta_{[nt]}-\xi_{[nt]})\\
    &+\sum_{k=0}^{[nt]-1}\sum_{\alpha+\beta=3,\alpha,\beta\in\mathbb{N}}R_2^{\alpha,\beta}(P_{t_k}^n,Q_{t_k}^n,\bar{\lambda}_1(t_{k+1}-t_k),\bar{\lambda}_2\Delta\hat{W}_k)(t_{k+1}-t_k)^{\alpha+\beta/2}(\zeta_{k}^\beta-\xi_k^\beta)\\
    &+\sum_{\alpha+\beta=3,\alpha,\beta\in\mathbb{N}}R_2^{\alpha,\beta}(P_{\eta_n(t)}^n,Q_{\eta_n(t)}^n,\bar{\lambda}_1(t-\eta_n(t)),\bar{\lambda}_2\Delta\hat{W}_{[nt],t})(t-\eta_n(t))^{\alpha+\beta/2}(\zeta_{[nt]}^\beta-\xi_{[nt]}^\beta).
\end{align*}

\subsection{Asymptotic error distribution of symplectic methods} \label{3.2}
Building on the preparation in \Cref{3.1}, we give the asymptotic error distribution of the symplectic Euler method in the following theorem, and subsequently extend the result to a class of symplectic methods for the $2d$-dimensional SHS (see \Cref{mainthmmulti}). The normalization constant is chosen as $\sqrt{n}$ since the strong convergence order of the class of symplectic methods we consider is 1/2 in this case.
\begin{thm} \label{mainthm}
    Let $U_P^n(t):=\sqrt{n}(P_t^n-P_t)$ and $U_Q^n(t):=\sqrt{n}(Q_t^n-Q_t)$, 
    where  $(P,Q)$ is the solution to \eqref{Generalequ} and  $(P^n,Q^n)$ is the continuous numerical solution corresponding to the symplectic Euler method. Then under Assumption \ref{assumption1}, we have $(U_P^n,U_Q^n)\Rightarrow^{stably} U:=(U_P,U_Q)$ in $\mathcal{C}([0,T];\mathbb{R}^{2})$ and $U$ satisfies 
\begin{align} \label{symEulerU}
        dU_P(t) =& \Big[(f+\frac{1}{2}a_1'a+\frac{1}{2}a_2'b)_1'(P_t,Q_t)U_P(t) +(f+\frac{1}{2}a_1'a+\frac{1}{2}a_2'b)_2'(P_t,Q_t)U_Q(t)\Big]dt \nonumber \\
        &+ \left[a_1'(P_t,Q_t)U_P(t)+a_2'(P_t,Q_t)U_Q(t)\right]dW_t+\frac{1}{\sqrt{2}} (a_1'a-a_2'b)(P_t,Q_t)dB_t, \nonumber \\
        dU_Q(t) =& \Big[(g+\frac{1}{2}b_1'a+\frac{1}{2}b_2'b)_1'(P_t,Q_t)U_P(t) +(g+\frac{1}{2}b_1'a+\frac{1}{2}b_2'b)_2'(P_t,Q_t)U_Q(t)\Big]dt \nonumber \\
        &+\Big[b_1'(P_t,Q_t)U_P(t)+b_2'(P_t,Q_t)U_Q(t)\Big]dW_t+\frac{1}{\sqrt{2}} (b_1'a-b_2'b)(P_t,Q_t)dB_t,
\end{align}
with initial value $(U_P(0),U_Q(0))=(0,0)$, where $B$ is a Brownian motion independent of $W$.
\end{thm}

Note that $U_P^n=\sqrt{n}(P^n-P)=\sqrt{n}(P^n-P)\mathbf{1}_{\Omega^n}+\sqrt{n}(P_0-P)\mathbf{1}_{(\Omega^n)^c}$. 
It follows from \Cref{propOmeganc} that
\begin{align*}
    \sqrt{n}\mathbb{E}\Big[ \sup_{t\in[0,T]}|P_0-P_t|\mathbf{1}_{(\Omega^n)^c} \Big] \le \sqrt{n}\Big(\mathbb{E}\Big[\sup_{t\in[0,T]}|P_0-P_t|^2\Big]\Big)^{1/2}(\mathbb{P}(\Omega^n)^c)^{1/2}\to0,
\end{align*}
which yields $\sqrt{n}(P_0-P)\mathbf{1}_{(\Omega^n)^c} \to^\mathbb{P} 0$ in $\mathcal{C}([0,T];\mathbb{R})$. 
This allows us to assume that $P_t^n$ and $Q_t^n$ satisfy equations \eqref{pn} and \eqref{qn} throughout $\Omega$ rather than only on $\Omega^n$, then it remains to show that $(U_P^n,U_Q^n) \Rightarrow (U_P,U_Q)$ in this setting. 

The proof of \Cref{mainthm} requires an analysis of the remainder terms in the integral forms \eqref{pn} and \eqref{qn} of $(P_t^n,Q_t^n)$. To this end, we now introduce the following proposition.

\begin{prop} \label{propconverge0}
    Let \Cref{assumption1} hold, then $R_P^n,R_Q^n \to^{\mathbb{P}} 0$ and $\sqrt{n}R_{P,2}^n,\sqrt{n}R_{P,3}^n,\sqrt{n}R_{Q,2}^n$, $\sqrt{n}R_{Q,3}^n\to^{\mathbb{P}}0$ in $\mathcal{C}([0,T];\mathbb{R})$ as $n\to \infty$.
\end{prop}
\begin{proof}
For $R_{P,1}^n$, applying the Burkholder--Davis--Gundy inequality and the H\"older inequality yields
\begin{align*}
    &\mathbb{E} \bigg[ \sup_{t\in [0,T]}\Big| \int_0^t(a_1'a)(P_{\eta_n(s)}^n,Q_{\eta_n(s)}^n)(W_s-W_{\eta_n(s)})dW_s \Big|^2 \bigg] \\
    \le &\ C\mathbb{E}\left[ \int_0^T \big|(a_1'a)(P_{\eta_n(s)}^n,Q_{\eta_n(s)}^n)\big|^2(W_s-W_{\eta_n(s)})^2ds\right] \\
    \le &\ C \int_0^T \Big(\mathbb{E}\Big[\big|(a_1'a)(P_{\eta_n(s)}^n,Q_{\eta_n(s)}^n)\big|^4\Big] \Big)^{1/2}\Big(\mathbb{E}\Big[(W_s-W_{\eta_n(s)})^4\Big]\Big)^{1/2}ds \\
    \le& \ C\int_0^T\Big(\mathbb{E}\Big[(W_s-W_{\eta_n(s)})^4\Big]\Big)^{1/2}ds 
    = C\int_0^T\big(s-\eta_n(s)\big)ds \to 0, \quad n\to \infty,
\end{align*}
where we use the fact that for any $q\ge1$, there exists $N\in \mathbb{N}^+$ and a constant $C(q)>0$ such that for all $n\ge N$,
   $ \sup_{s\in[0,T]}\mathbb{E}[|a_1'a(P_{\eta_n(s)}^n,Q_{\eta_n(s)}^n)|^q] 
   \le C(q).$
This follows from the polynomial growth of $a_1'a$ and the uniform moment boundedness $\sup_{k=0,\dots,nT} \mathbb{E}[\|(P_{t_k}^n,Q_{t_k}^n)\|^{q_0}]\le C(q_0)$ for any $q_0\geq 1$ and some constant $C(q_0)>0$.

For $\sqrt{n}R_{P,2}^n$, each summation term in its expression can be rewritten in integral form via It\^o formula. For example, it follows from
\begin{align*}
    (t-s)(W_t-W_s) = \int_s^t (\tau-s)dW_\tau + \int_s^t (W_\tau-W_s)d\tau
\end{align*}
that
\begin{align} \label{sumtointegral}
        &\sqrt{n}\sum_{k=0}^{[nt]-1}(u_1'a+a_1'u)(P_{t_k}^n,Q_{t_k}^n)(t_{k+1}-t_k)(W_{t_{k+1}}-W_{t_k}) \nonumber \\
        &+\sqrt{n}(u_1'a+a_1'u)(P_{\eta_n(t)}^n,Q_{\eta_n(t)}^n)(t-\eta_n(t))(W_t-W_{\eta_n(t)}) \nonumber \\
        =&\ \sqrt{n}\int_0^t(u_1'a+a_1'u)(P_{\eta_n(s)}^n,Q_{\eta_n(s)}^n)(s-\eta_n(s))dW_s \nonumber \\
        &+\sqrt{n}\int_0^t (u_1'a+a_1'u)(P_{\eta_n(s)}^n,Q_{\eta_n(s)}^n)(W_s-W_{\eta_n(s)})ds.
\end{align}
In fact, the second term on the right-hand side of \eqref{sumtointegral} has an equivalent integral representation, specifically, 
\begin{align*}
    &\sqrt{n}\int_0^t (u_1'a+a_1'u)(P_{\eta_n(s)}^n,Q_{\eta_n(s)}^n)(W_s-W_{\eta_n(s)})ds \\=&\, \sqrt{n}\int_0^t (u_1'a+a_1'u)(P_{\eta_n(s)}^n,Q_{\eta_n(s)}^n)\int_{\eta_n(s)}^sdW_\tau ds \\
    =&\, \sqrt{n}\int_0^t\int_\tau^{\eta_n(\tau)+1/n}(u_1'a+a_1'u)(P_{\eta_n(s)}^n,Q_{\eta_n(s)}^n)dsdW_\tau\\
    =&\, \sqrt{n}\int_0^t(u_1'a+a_1'u)(P_{\eta_n(\tau)}^n,Q_{\eta_n(\tau)}^n)(\eta_n(\tau)+\frac{1}{n}-\tau)dW_\tau.
\end{align*}
Thus, by a similar analysis as for $R_{P,1}^n$, we have
\begin{align*}
    &n\mathbb{E} \bigg[ \sup_{t\in [0,T]}\Big| \int_0^t(u_1'a+a_1'u)(P_{\eta_n(s)}^n,Q_{\eta_n(s)}^n)(s-\eta_n(s))dW_s \Big|^2 \bigg] \\
    +& n\mathbb{E} \bigg[ \sup_{t\in [0,T]}\Big| \int_0^t(u_1'a+a_1'u)(P_{\eta_n(s)}^n,Q_{\eta_n(s)}^n)(W_s-W_{\eta_n(s)})ds \Big|^2 \bigg]
    \le Cn\int_0^T (s-\eta_n(s))^2ds \to 0.
\end{align*}
As a result, $\sqrt{n}R_{P,2}^n\to^{\mathbb{P}}0$ in $\mathcal{C}([0,T];\mathbb{R})$, 
as the other terms in the expression of $\sqrt{n}R_{P,2}^n$ can be shown to vanish in the same way. 

For $\sqrt{n}R_{P,3}^n$, we next show that the first term converges to 0 in probability in $\mathcal{C}([0,T];\mathbb{R})$ as an illustrative case. Combining \Cref{assumption1} and \eqref{truncated1} leads to
\begin{align*}
    &\sqrt{n}\,\mathbb{E}\bigg[ \sup_{t\in[0,T]}\Big|\sum_{k=0}^{[nt]-1}a(P_{t_k}^n,Q_{t_k}^n)\frac{1}{\sqrt{n}}(\zeta_{k}-\xi_k)+a(P_{\eta_n(t)}^n,Q_{\eta_n(t)}^n)\sqrt{t-\eta_n(t)}(\zeta_{[nt]}-\xi_{[nt]})\Big| \bigg] \\
    \le &\,  \sqrt{n}\,\mathbb{E}\bigg[ \sup_{t\in[0,T]} \frac{1}{\sqrt{n}}\sum_{k=0}^{[nt]}\big|a(P_{t_k}^n,Q_{t_k}^n)\big|\big|\zeta_{k}-\xi_k\big| \bigg]
    \le C\mathbb{E}\bigg[ \sum_{k=0}^{nT}\big|a(P_{t_k}^n,Q_{t_k}^n)\big|\big|\zeta_{k}-\xi_k\big| \bigg]\\
    \le & \, C\sum_{k=0}^{nT}\left( \mathbb{E}\big[a^2(P_{t_k}^n,Q_{t_k}^n)\big] \right)^{1/2}\left( \mathbb{E}|\zeta_{k}-\xi_k|^2 \right)^{1/2}
    \le C(nT+1)\big(\frac{1}{n}\big)^{\rho/2},
\end{align*}
which therefore tends to 0 due to $\rho>2$.
Similarly, the remainder terms converge to 0 by \eqref{truncated1} and \eqref{truncated2}. 

Since $R_P^n=R_{P,1}^n+R_{P,2}^n+R_{P,3}^n$, we conclude that $R_P^n \to^{\mathbb P} 0$ in $\mathcal{C}([0,T];\mathbb{R})$.  
Using the preceding argument, it can be analogously proved that $R_{Q}^n,\sqrt{n}R_{Q,2}^n,\sqrt{n}R_{Q,3}^n\to^{\mathbb{P}}0$ in $\mathcal{C}([0,T];\mathbb{R})$. 
\end{proof}

With previous preparation, we now give the proof of \Cref{mainthm}.
\begin{proof}[Proof of \Cref{mainthm}]

Given that \eqref{pn}, \eqref{qn} hold and $R_P^n, R_Q^n \to^{\mathbb{P}}0$ in $\mathcal{C}([0,T];\mathbb{R})$, it follows from \Cref{propmain} that $(P^n, Q^n, W) \to^{\mathbb{P}} (P, Q, W)$ as $n \to \infty$. Consequently, we also have $(P^n, Q^n, W) \Rightarrow (P, Q, W)$.
By \eqref{p} and \eqref{pn},
\begin{align*}
    \sqrt{n}(P_t^n-P_t) =& \,\sqrt{n}\int_0^t [(u+a_1'a)(P_{\eta_n(s)}^n,Q_{\eta_n(s)}^n)-(u+a_1'a)(P_s,Q_s)]ds \\
    &+\sqrt{n} \int_0^t [a(P_{\eta_n(s)}^n,Q_{\eta_n(s)}^n)-a(P_s,Q_s)]dW_s + \sqrt{n}R_P^n(t) \\
    =:& \,I_1^n(t)+I_2^n(t)+\sqrt{n}R_P^n(t).
\end{align*}

\textit{(\romannumeral1) Reformulation of $I_1^n$.} In fact,
\begin{align*}
    I_1^n(t) =& \,\sqrt{n}\int_0^t [(u+a_1'a)(P_s^n,Q_s^n)-(u+a_1'a)(P_s,Q_s)]ds \\
    &- \sqrt{n}\int_0^t [(u+a_1'a)(P_s^n,Q_s^n)-(u+a_1'a)(P_{\eta_n(s)}^n,Q_{\eta_n(s)}^n)]ds \\
    =& \int_0^t [(u+a_1'a)_1'(P_s,Q_s)U_P^n(s)+(u+a_1'a)_2'(P_s,Q_s)U_Q^n(s)]ds -I_{11}^n(t)-I_{12}^n(t)+I_{1_R}^n(t),
\end{align*}
where
\begin{align*}
    &I_{11}^n(t):= \sqrt{n}\int_0^t (u+a_1'a)_1'(P_{\eta_n(s)}^n,Q_{\eta_n(s)}^n)(P_s^n-P_{\eta_n(s)}^n)ds,\\
    &I_{12}^n(t):= \sqrt{n}\int_0^t (u+a_1'a)_2'(P_{\eta_n(s)}^n,Q_{\eta_n(s)}^n)(Q_s^n-Q_{\eta_n(s)}^n)ds,
\end{align*}
and 
\begin{align} \label{I1Rn}
    I_{1_R}^n(t):=&\  \sqrt{n} \int_0^t \int_0^1 (1-\lambda)(X_s^n-X_s)^{\top} D^2(u+a_1'a)(\Theta_1(\lambda,s))(X_s^n-X_s)d\lambda ds \nonumber \\
    &- \sqrt{n} \int_0^t \int_0^1 (1-\lambda)(X_s^n-X_{\eta_n(s)}^n)^{\top} D^2(u+a_1'a)(\Theta_2(\lambda,s))(X_s^n-X_{\eta_n(s)}^n)d\lambda ds,
\end{align}
with $X^n:=(P^n,Q^n)^{\top}$, $X:=(P,Q)^{\top}$, $\Theta_1(\lambda,s):=X_s+\lambda(X_s^n-X_s)$, $\Theta_2(\lambda,s):=X_{\eta_n(s)}^n+\lambda(X_s^n-X_{\eta_n(s)}^n)$,
and $D^2(u+a_1'a)$ denoting the Hessian matrix of $u+a_1'a$.
Plugging \eqref{explicitp} into $I_{11}^n$ gives
\begin{align*}
    I_{11}^n(t) =&\  \sqrt{n} \int_0^t ((u+a_1'a)_1'u)(P_{\eta_n(s)}^n,Q_{\eta_n(s)}^n)(s-\eta_n(s))ds \\
    &+\sqrt{n}\int_0^t ((u+a_1'a)_1'a)(P_{\eta_n(s)}^n,Q_{\eta_n(s)}^n)(W_s-W_{\eta_n(s)})ds \\
    &+\sqrt{n}\int_0^t ((u+a_1'a)_1'a)(P_{\eta_n(s)}^n,Q_{\eta_n(s)}^n)(s-\eta_n(s))^{1/2}(\zeta_{[ns]}-\xi_{[ns]})ds \\
    &+\sum_{\alpha+\beta=2,3,\alpha,\beta\in\mathbb{N}} \sqrt{n}\int_0^t F_{\alpha,\beta}(P_{\eta_n(s)}^n,Q_{\eta_n(s)}^n)(s-\eta_n(s))^\alpha(\Delta \hat{W}_{[ns],s})^\beta ds,
\end{align*}
with $F_{\alpha,\beta}$ polynomially growing. By a similar argument to \Cref{propconverge0}, we arrive at $I_{11}^n\to^{\mathbb{P}}0$ and  $I_{11}^n \Rightarrow 0$ in $\mathcal{C}([0,T];\mathbb{R})$ as $n\to \infty$.
We also obtain $I_{12}^n \Rightarrow 0$ in $\mathcal{C}([0,T];\mathbb{R})$  by plugging \eqref{explicitq} into $I_{12}^n$.
For $I_{1_R}^n$, it follows from the polynomial growth of $D^2(u+a_1'a)$ that 
\begin{align*}
    &\ n\mathbb{E}\bigg[ \sup_{t\in[0,T]} \Big( \int_0^t\int_0^1(1-\lambda)(X^n_s-X_s)^{\top}D^2(u+a_1'a)(\Theta_1(\lambda,s))(X^n_s-X_s)d\lambda ds \Big)^2 \bigg] \\
   \le &\  nT\,\mathbb{E}\bigg[ \int_0^T \int_0^1(1-\lambda)^2\Big((X^n_s-X_s)^{\top}D^2(u+a_1'a)(\Theta_1(\lambda,s))(X^n_s-X_s)\Big)^2d\lambda ds \bigg] \\
    \le &\ nC\,\mathbb{E}\bigg[ \int_0^T\big(\|X^n_s-X_s\|^4+\|X_s\|^{2\gamma}\|X^n_s-X_s\|^4+\|X^n_s-X_s\|^{2\gamma+4}\big)ds \bigg] \to 0,
\end{align*}
since for any $q\ge1$, there exists $N\in \mathbb{N}^+$ and constant $C>0$ such that for all $n\ge N$,
\begin{align*}
    \sup_{t\in[0,T]}\big(\mathbb{E}[\|X_t^n-X_t\|^{2q}]\big)^{1/2q} \le C(\frac{1}{n})^{1/2}.
\end{align*} 
In the same manner, the fact that 
\begin{align*}
    \sup_{t\in [0,T]}\big(\mathbb{E}[\|X_t^n-X_{\eta_n(t)}^n\|^{2q}]\big)^{1/2q} \le C(\frac{1}{n})^{1/2}
\end{align*}
leads to the estimate for the second term of $I_{1_R}^n$ in \eqref{I1Rn} as follows:
\begin{align*}
    &n\mathbb{E}\bigg[ \sup_{t\in[0,T]} \Big( \int_0^t\int_0^1(1-\lambda)(X^n_s-X_{\eta_n(s)}^n)^{\top}D^2(u+a_1'a)(\Theta_2(\lambda,s))(X^n_s-X_{\eta_n(s)}^n)d\lambda ds \Big)^2 \bigg] \\
    \le &nC\mathbb{E}\bigg[ \int_0^T\big(\|X^n_s-X_{\eta_n(s)}^n\|^4+\|X_{\eta_n(s)}^n\|^{2\gamma}\|X^n_s-X_{\eta_n(s)}^n\|^4+\|X^n_s-X_{\eta_n(s)}^n\|^{2\gamma+4}\big)ds \bigg] \to 0.
\end{align*}
Thus, by setting $M_1^n:=-I_{11}^n-I_{12}^n+I_{1_R}^n$, we conclude that
\begin{align*}
    I_1^n(t) = \int_0^t [(u+a_1'a)_1'(P,Q)U_P^n+(u+a_1'a)_2'(P,Q)U_Q^n]ds +M_{1}^n(t),
\end{align*}
where $M_{1}^n\Rightarrow0$ in $\mathcal{C}([0,T];\mathbb{R})$ as $n\to \infty$.

\textit{(\romannumeral2) Reformulation of $I_2^n$.} Observe that
\begin{align*}
    I_2^n(t) =& \,\sqrt{n}\int_0^t [a(P_s^n,Q_s^n)-a(P_s,Q_s)]dW_s - \sqrt{n}\int_0^t [a(P_s^n,Q_s^n)-a(P_{\eta_n(s)}^n,Q_{\eta_n(s)}^n)]dW_s \\
    =& \int_0^t [a_1'(P,Q)U_P^n+a_2'(P,Q)U_Q^n]dW_s -I_{21}^n(t)-I_{22}^n(t)+ I_{2_R}^n(t),
\end{align*}
where 
\begin{align*}
    &I_{21}^n(t):=\sqrt{n}\int_0^t a_1'(P_{\eta_n(s)}^n,Q_{\eta_n(s)}^n)(P_s^n-P_{\eta_n(s)}^n)dW_s,\\
    &I_{22}^n(t):=\sqrt{n}\int_0^t a_2'(P_{\eta_n(s)}^n,Q_{\eta_n(s)}^n)(Q_s^n-Q_{\eta_n(s)}^n)dW_s,
\end{align*}
and
\begin{align} \label{I2Rn}
    I_{2_R}^n(t):=&\, \sqrt{n} \int_0^t \int_0^1 (1-\lambda)(X_s^n-X_s)^{\top}D^2a(\Theta_1(\lambda,s))(X_s^n-X_s)d\lambda dW_s \nonumber \\
    &- \sqrt{n} \int_0^t \int_0^1 (1-\lambda)(X_s^n-X_{\eta_n(s)}^n)^{\top}D^2a(\Theta_2(\lambda,s))(X_s^n-X_{\eta_n(s)}^n)d\lambda dW_s.
\end{align}
Plugging \eqref{explicitp} into $I_{21}^n$ yields
\begin{equation} \label{I_21^n}
    \begin{split}
        I_{21}^n(t) = \sqrt{n}\int_0^t (a_1'a)(P_{\eta_n(s)}^n,Q_{\eta_n(s)}^n)(W_s-W_{\eta_n(s)})dW_s +I_{21_R}^n(t),
    \end{split}
\end{equation}
where 
\begin{align*}
    I_{21_R}^n(t):= &\, \sqrt{n} \int_0^t (a_1'u)(P_{\eta_n(s)}^n,Q_{\eta_n(s)}^n)(s-\eta_n(s))dW_s \\
    &+ \sqrt{n}\int_0^t (a_1'a)(P_{\eta_n(s)}^n,Q_{\eta_n(s)}^n)(s-\eta_n(s))^{1/2}(\zeta_{[ns]}-\xi_{[ns]})dW_s \\
    &+\sum_{\alpha+\beta=2,3,\alpha,\beta\in\mathbb{N}} \sqrt{n}\int_0^t G_{\alpha,\beta}(P_{\eta_n(s)}^n,Q_{\eta_n(s)}^n)(s-\eta_n(s))^\alpha(\Delta \hat{W}_{[ns],s})^\beta dW_s,
\end{align*}
with $G_{\alpha,\beta}$ polynomially growing. 
By an argument analogous to that in \Cref{propconverge0}, we can prove that $I_{21_R}^n$ converges to 0 in probability in $\mathcal{C}([0,T];\mathbb{R})$, and hence also converges to 0 in distribution. For the first term of $I_{21}^n$ in \eqref{I_21^n}, let $Y_n(t):=\sqrt{n} \int_0^t(W_s-W_{\eta_n(s)})dW_s$, then
\begin{align*}
    \sqrt{n}\int_0^t (a_1'a)(P_{\eta_n(s)}^n,Q_{\eta_n(s)}^n)(W_s-W_{\eta_n(s)})dW_s=\int_0^t (a_1'a)(P_{\eta_n(s)}^n,Q_{\eta_n(s)}^n)dY_n(s).
\end{align*}
Applying \cite[Theorem 2.1]{Rootzen1980} gives $Y_n\Rightarrow \frac{1}{\sqrt{2}}B$, where $B$ is a Brownian motion independent of $W$. Combining this with the convergence $(P^n,Q^n)\Rightarrow(P,Q)$ yields $((a_1'a)(P^n,Q^n),Y_n)\Rightarrow((a_1'a)(P,Q),\frac{1}{\sqrt{2}}B)$, which relies on the temporal continuity of the solution process $(P, Q)$ and the Brownian motion $B$, as well as the continuity of the function $a_1'a$.
In addition, based on the goodness of the sequence $\{Y_n\}$ verified via \cite[Theorem 2.7]{Burdzy2010} and applying \cite[Lemma 3.2]{KurtzProtter1991b}, we obtain
\begin{align*}
    \int_0^{\cdot} (a_1'a)(P_{\eta_n(s)}^n,Q_{\eta_n(s)}^n)dY_n(s) \Rightarrow \frac{1}{\sqrt{2}}\int_0^{\cdot} (a_1'a)(P_s,Q_s)dB_s
\end{align*}
in $\mathcal{C}([0,T];\mathbb{R})$. 
In conclusion, $I_{21}^n\Rightarrow \frac{1}{\sqrt{2}}\int_0^{\cdot} (a_1'a)(P_s,Q_s)dB_s$ in $\mathcal{C}([0,T];\mathbb{R})$. Similarly, we arrive at $I_{22}^n \Rightarrow \frac{1}{\sqrt{2}} \int_0^{\cdot} (a_2'b)(P_s,Q_s)dB_s$ in $\mathcal{C}([0,T];\mathbb{R})$.
For $I_{2_R}^n$, 
by the Burkholder--Davis--Gundy inequality, 
\begin{align*}
    &\ n\mathbb{E}\bigg[ \sup_{t\in[0,T]} \Big| \int_0^t\int_0^1(1-\lambda)(X^n_s-X_s)^{\top}D^2a(\Theta_1(\lambda,s))(X^n_s-X_s)d\lambda dW_s \Big|^2 \bigg] \\
    \le &\  nC\,\mathbb{E}\bigg[  \int_0^T \Big( \int_0^1 (1-\lambda)(X^n_s-X_s)^{\top}D^2a(\Theta_1(\lambda,s))(X^n_s-X_s)d\lambda \Big)^2ds \bigg] \\
    \le&\ nC\,\mathbb{E}\bigg[ \int_0^T(\|X^n_s-X_s\|^4+\|X_s\|^{2\gamma}\|X^n_s-X_s\|^4+\|X^n_s-X_s\|^{2\gamma+4})ds \bigg] \to 0.
\end{align*}
Moreover, we similarly prove that the second term of $I_{2_R}^n$ in \eqref{I2Rn}  converges to 0. Hence, $I_{2_R}^n\Rightarrow0$ as $n\to \infty$.
Combining the estimates above leads to
\begin{align*}
    I_2^n(t) = \int_0^t [a_1'(P_s,Q_s)U_P^n(s)+a_2'(P_s,Q_s)U_Q^n(s)]dW_s + M_{2}^n(t),
\end{align*}
where $M_2^n:=-I_{21}^n-I_{22}^n+I_{2_R}^n$ satisfies
\begin{align*}
    M_{2}^n\Rightarrow -\frac{1}{\sqrt{2}}\int_0^{\cdot}(a_1'a+a_2'b)(P_s,Q_s)dB_s.
\end{align*}

\textit{(\romannumeral3) Convergence of $\sqrt{n}R_P^n$.}
Given that $\sqrt{n}R_{P,2}^n$ and $\sqrt{n}R_{P,3}^n$ converge to 0 by \Cref{propconverge0}, and from the decomposition $\sqrt{n}R_P^n = \sqrt{n}R_{P,1}^n + \sqrt{n}R_{P,2}^n + \sqrt{n}R_{P,3}^n$, the convergence of $\sqrt{n}R_P^n$ reduces to that of $\sqrt{n}R_{P,1}^n$. By a similar argument, it is clear that 
$$\sqrt{n}R_{P,1}^n=2\sqrt{n}\int_0^{\cdot} (a_1'a)(P_{\eta_n(s)}^n,Q_{\eta_n(s)}^n)(W_s-W_{\eta_n(s)})dW_s \Rightarrow \sqrt{2} \int_0^{\cdot} (a_1'a)(P_s,Q_s)dB_s.$$

Combining \textit{(\romannumeral 1)-(\romannumeral 3)}, 
$U_P^n$ satisfies
\begin{align*}
    U_P^n(t) =& \int_0^t [(u+a_1'a)_1'(P_s,Q_s)U_P^n(s)+(u+a_1'a)_2'(P_s,Q_s)U_Q^n(s)]ds \\
    &+\int_0^t [a_1'(P_s,Q_s)U_P^n(s)+a_2'(P_s,Q_s)U_Q^n(s)]dW_s + T_1^n(t),
\end{align*}
where $T_1^n:=M_1^n+M_2^n+\sqrt{n}R_{P,1}^n \Rightarrow \frac{1}{\sqrt{2}} \int_0^{\cdot} (a_1'a-a_2'b)(P_s,Q_s)dB_s$.
A similar analysis applied to $U_Q^n$ shows that
\begin{align*}
    U_Q^n(t) =& \int_0^t [(v+b_1'a)_1'(P_s,Q_s)U_P^n(s)+(v+b_1'a)_2'(P_s,Q_s)U_Q^n(s)]ds \\
    &+\int_0^t [b_1'(P_s,Q_s)U_P^n(s)+b_2'(P_s,Q_s)U_Q^n(s)]dW_s + T_2^n(t),
\end{align*}
where $T_2^n \Rightarrow \frac{1}{\sqrt{2}} \int_0^{\cdot} (b_1'a-b_2'b)(P_s,Q_s)dB_s$. 
Hence, by \Cref{rmk1}, it follows that $(U_P^n,U_Q^n,W) \Rightarrow (U_P,U_Q,W)$, where $(U_P,U_Q)$ satisfies \eqref{symEulerU}. 
Furthermore, 
\Cref{lemmastable}(\romannumeral2) yields $(U_P^n,U_Q^n)\Rightarrow^{stably} U:=(U_P,U_Q)$.
\end{proof}

It should be noted that the technique we used to calculate the asymptotic error distribution is also applicable for multi-dimensional systems. 
To be specific, we consider the $2d$-dimensional SHS
\begin{align} \label{Generalequ2d}
    d \begin{pmatrix} P_t \\ Q_t
\end{pmatrix}=\begin{pmatrix}
    f(P_t,Q_t) \\
    g(P_t,Q_t)
\end{pmatrix}dt +\begin{pmatrix}
    a(P_t,Q_t) \\ b(P_t,Q_t)
\end{pmatrix} \circ dW_t, \quad t\in (0,T],
\end{align}
with the initial value $(P_0,Q_0)\in \mathbb{R}^{2d}$, where $f:=-\frac{\partial H}{\partial Q}$, $g:=\frac{\partial H}{\partial P}$, $a:=-\frac{\partial \bar{H}}{\partial Q}$, $b:=\frac{\partial \bar{H}}{\partial P}$ and $H,\bar{H}:\mathbb{R}^{2d}\to \mathbb{R}$ are Hamiltonian functions. We apply a class of symplectic methods (see \cite[Eq.~(5.2.7)]{Milstein2021}) 
to \eqref{Generalequ2d} and give its continuous version as follows:
\begin{align} \label{generalsym}
        P_t^n =& \,P_{t_k}^n + \Big[f+(\frac{1}{2}-\theta)(\frac{\partial a}{\partial P}a-\frac{\partial a}{\partial Q}b)\Big](\theta P_t^n+(1-\theta)P_{t_k}^n, (1-\theta)Q_t^n+\theta Q_{t_k}^n)(t-t_k) \nonumber \\
        &+ a(\theta P_t^n+(1-\theta)P_{t_k}^n, (1-\theta)Q_t^n+\theta Q_{t_k}^n)\Delta\hat{W}_{k,t} \quad t\in (t_k,t_{k+1}], \nonumber \\
        Q_t^n =&\, Q_{t_k}^n +\Big[g+(\frac{1}{2}-\theta)(\frac{\partial b}{\partial P}a-\frac{\partial b}{\partial Q}b)\Big](\theta P_t^n+(1-\theta)P_{t_k}^n, (1-\theta)Q_t^n+\theta Q_{t_k}^n)(t-t_k) \nonumber \\
        &+b(\theta P_t^n+(1-\theta)P_{t_k}^n, (1-\theta)Q_t^n+\theta Q_{t_k}^n)\Delta\hat{W}_{k,t} \quad t\in (t_k,t_{k+1}],
\end{align} 
for $k=0,\dots,nT-1$ with $\theta \in [0,1]$. The symplectic Euler method is included in \eqref{generalsym} as a special case with $\theta = 1$.
By extending \Cref{assumption1} to the $2d$-dimensional case, we next establish the asymptotic error distributions for this class of symplectic methods.
\begin{assumption} \label{assumption2}
    (\romannumeral1) $f,g,a,b,\frac{\partial a}{\partial P}a,\frac{\partial a}{\partial Q}b,\frac{\partial b}{\partial P}a,\frac{\partial b}{\partial Q}b:\mathbb{R}^{2d}\to \mathbb{R}^d$ are globally Lipschitz \mbox{continuous}.\\
    (\romannumeral2) Let $f,g\in C^3(\mathbb{R}^{2d};\mathbb{R}^d)$, $a,b\in C^4(\mathbb{R}^{2d};\mathbb{R}^d)$ and assume that the $k$th order derivatives $D^k f,D^k g$ for $k= 2,3$ and $D^k a,D^k b$ for $k= 2,3,4$ are functions growing at most \mbox{polynomially}.
\end{assumption}

\begin{thm} \label{mainthmmulti}
Let $U_P^n(t):=\sqrt{n}(P_t^n-P_t)$ and $U_Q^n(t):=\sqrt{n}(Q_t^n-Q_t)$, where  $(P,Q)$ is the solution to \eqref{Generalequ2d} and  $(P^n,Q^n)$ is defined by \eqref{generalsym}. Then under Assumption \ref{assumption2},  we have $(U_P^n,U_Q^n)\Rightarrow^{stably} U:=(U_P,U_Q)$ in $\mathcal{C}([0,T];\mathbb{R}^{2d})$ and $U$ satisfies 
\begin{align} \label{generalU}
        dU_P(t) 
        =& \,\Big[\frac{\partial}{\partial P}\Big(f+\frac{1}{2}\frac{\partial a}{\partial P}a+\frac{1}{2}\frac{\partial a}{\partial Q}b\Big)(P_t,Q_t)U_P(t) +\frac{\partial}{\partial Q}\Big(f+\frac{1}{2}\frac{\partial a}{\partial P}a+\frac{1}{2}\frac{\partial a}{\partial Q}b\Big)(P_t,Q_t)U_Q(t)\Big]dt \nonumber \\
        &+ \Big[\frac{\partial a}{\partial P}(P_t,Q_t)U_P(t)+\frac{\partial a}{\partial Q}(P_t,Q_t)U_Q(t)\Big]dW_t+\frac{2\theta-1}{\sqrt{2}} \Big(\frac{\partial a}{\partial P}a-\frac{\partial a}{\partial Q}b\Big)(P_t,Q_t)dB_t, \nonumber \\
        dU_Q(t)
        =& \,\Big[\frac{\partial}{\partial P}\Big(g+\frac{1}{2}\frac{\partial b}{\partial P}a+\frac{1}{2}\frac{\partial b}{\partial Q}b\Big)(P_t,Q_t)U_P(t) +\frac{\partial }{\partial Q}\Big(g+\frac{1}{2}\frac{\partial b}{\partial P}a+\frac{1}{2}\frac{\partial b}{\partial Q}b\Big)(P_t,Q_t)U_Q(t)\Big]dt \nonumber \\
        &+ \Big[\frac{\partial b}{\partial P}(P_t,Q_t)U_P(t)+\frac{\partial b}{\partial Q}(P_t,Q_t)U_Q(t)\Big]dW_t+\frac{2\theta-1}{\sqrt{2}} \Big(\frac{\partial b}{\partial P}a-\frac{\partial b}{\partial Q}b\Big)(P_t,Q_t)dB_t,
\end{align}
with initial value $(U_P(0),U_Q(0))=(0,0)$, where $B$ is a Brownian motion independent of $W$.
\end{thm}

\begin{remark}
When $\theta = \frac{1}{2}$, the term involving $B$ in \eqref{generalU} vanishes, and thus $U = 0$ becomes a solution to \eqref{generalU}. This is consistent with the fact that \eqref{generalsym} with $\theta = \frac{1}{2}$ is the midpoint method, which has a strong convergence order of 1. In this case, the normalized error process should be defined using a normalization constant $n$ rather than $\sqrt{n}$.
\end{remark}

\subsection{Hamiltonian-specific results} \label{3.3}
Building on the derived asymptotic error distribution, the following theorem shows that the equation of $U$ retains a stochastic Hamiltonian formulation.
\begin{thm} \label{UisHamiltonian}
The asymtotic error distribution $U:=(U_P,U_Q)$  given in \eqref{generalU} still has a Hamiltonian formulation, satisfying
\begin{equation*} 
    \begin{split}
        d \begin{pmatrix}
            U_P \\
            U_Q
        \end{pmatrix} =& \begin{pmatrix}
            -\frac{\partial H_0}{\partial U_Q} \\
            \frac{\partial H_0}{\partial U_P}
        \end{pmatrix} dt+\begin{pmatrix}
            -\frac{\partial H_1}{\partial U_Q} \\
            \frac{\partial H_0}{\partial U_P}
        \end{pmatrix} \circ dW_t + \begin{pmatrix}
            -\frac{\partial H_2}{\partial U_Q} \\
            \frac{\partial H_0}{\partial U_P}
        \end{pmatrix} \circ dB_t,
    \end{split}
\end{equation*}
where 
\begin{align}
    &H_0 = \frac{1}{2}U_P^{\top}\frac{\partial g}{\partial P}(P,Q)U_P-U_P^{\top}\frac{\partial f}{\partial P}(P,Q)U_Q-\frac{1}{2}U_Q^{\top}\frac{\partial f}{\partial Q}(P,Q)U_Q, \label{H0} \\
    &H_1 = \frac{1}{2}U_P^{\top}\frac{\partial b}{\partial P}(P,Q)U_P-U_P^{\top}\frac{\partial a}{\partial P}(P,Q)U_Q-\frac{1}{2}U_Q^{\top}\frac{\partial a}{\partial Q}(P,Q)U_Q, \label{H1} \\
    &H_2 = \frac{2\theta-1}{\sqrt{2}} U_P^{\top}\left(\frac{\partial b}{\partial P}a-\frac{\partial b}{\partial Q}b\right)(P,Q)-\frac{2\theta-1}{\sqrt{2}} U_Q^{\top}\left(\frac{\partial a}{\partial P}a-\frac{\partial a}{\partial Q}b\right)(P,Q). \label{H2}
\end{align} 
\end{thm}
\begin{proof}
By \eqref{Generalequ2d} and the It\^o--Stratonovich conversion formula, we rewrite \eqref{generalU} as
{\small
\begin{align} \label{weifenStratonovichU}
    d \begin{pmatrix}
        U_P(t) \\
        U_Q(t)
    \end{pmatrix} 
    =& \begin{pmatrix}
        \frac{\partial f}{\partial P}(P_t,Q_t)U_P(t)+\frac{\partial f}{\partial Q}(P_t,Q_t)U_Q(t) \\
        \frac{\partial g}{\partial P}(P_t,Q_t)U_P(t)+\frac{\partial g}{\partial Q}(P_t,Q_t)U_Q(t)
    \end{pmatrix} dt +\begin{pmatrix}
        \frac{\partial a}{\partial P}(P_t,Q_t)U_P(t)+\frac{\partial a}{\partial Q}(P_t,Q_t)U_Q(t) \\
        \frac{\partial b}{\partial P}(P_t,Q_t)U_P(t)+\frac{\partial b}{\partial Q}(P_t,Q_t)U_Q(t)
    \end{pmatrix} \circ dW_t \nonumber \\
    &+ \begin{pmatrix}
        \frac{2\theta-1}{\sqrt{2}} \left(\frac{\partial a}{\partial P}a-\frac{\partial a}{\partial Q}b\right)(P_t,Q_t) \\
        \frac{2\theta-1}{\sqrt{2}} \left(\frac{\partial b}{\partial P}a-\frac{\partial b}{\partial Q}b\right)(P_t,Q_t)
        \end{pmatrix} \circ dB_t.
\end{align}
}
\\
Since \eqref{Generalequ2d} is an SHS and the relations $\frac{\partial f}{\partial P}=-\frac{\partial g}{\partial Q}$ and $\frac{\partial a}{\partial P}=-\frac{\partial b}{\partial Q}$ hold, we conclude that \eqref{weifenStratonovichU} is also an SHS, with Hamiltonians given by \eqref{H0}-\eqref{H2}.
\end{proof}

Next, we show the limiting distribution of the normalized Hamiltonian deviation.
\begin{thm} \label{thmdeltaH}
    If $H$ is one of the Hamiltonians in \eqref{Generalequ2d} and $U=(U_P,U_Q)$ is the asymptotic error distribution given in \eqref{generalU}, then 
    \begin{align} \label{multideltaH}
        \sqrt{n}(H(P^n,Q^n)-H(P,Q))\Rightarrow^{stably}  \Big(\frac{\partial H}{\partial P}(P,Q)\Big)^{\top}U_P+\Big(\frac{\partial H}{\partial Q}(P,Q)\Big)^{\top}U_Q
    \end{align}
    in $\mathcal C ([0,T];{\mathbb{R}})$.
\end{thm}
\begin{proof}
    By the Taylor formula,  we obtain
    \begin{align*}
        \sqrt{n}(H(P^n,Q^n)-H(P,Q))=& \ \sqrt{n}\Big(\frac{\partial H}{\partial P}(P,Q)\Big)^{\top} (P^n-P)+\sqrt{n}\Big(\frac{\partial H}{\partial Q}(P,Q)\Big)^{\top}(Q^n-Q) \\
        &+ \sqrt{n}\begin{pmatrix}
        P^n-P \\
        Q^n-Q
    \end{pmatrix}^{\top}D^2H(\Theta(\lambda))\begin{pmatrix}
        P^n-P \\
        Q^n-Q
    \end{pmatrix},
    \end{align*}
    where $\Theta(\lambda)=(P+\lambda(P^n-P),Q+\lambda(Q^n-Q))$ with $\lambda \in [0,1]$. Since $(\sqrt{n}(P^n-P),\sqrt{n}(Q^n-Q))\Rightarrow^{stably}(U_P,U_Q)$, by \Cref{lemmastable}(i), we deduce that 
    \begin{align*}
        &\quad \sqrt{n}\Big(\frac{\partial H}{\partial P}(P,Q)\Big)^{\top}(P^n-P)+\sqrt{n}\Big(\frac{\partial H}{\partial Q}(P,Q)\Big)^{\top}(Q^n-Q) \\ &\Rightarrow^{stably}\Big(\frac{\partial H}{\partial P}(P,Q)\Big)^{\top}U_P+\Big(\frac{\partial H}{\partial Q}(P,Q)\Big)^{\top}U_Q
    \end{align*}
in $\mathcal C ([0,T];{\mathbb{R}})$. Furthermore, the boundedness of $D^2H$ and the fact that $\big(\begin{smallmatrix} P^n-P \\ Q^n-Q \end{smallmatrix} \big) \to^{\mathbb{P}} 0$ lead to $D^2H(\Theta(\lambda))\big(\begin{smallmatrix} P^n-P \\ Q^n-Q \end{smallmatrix} \big) \to^{\mathbb{P}}0$. Therefore, we have $\sqrt{n}\big(\begin{smallmatrix} P^n-P \\ Q^n-Q \end{smallmatrix} \big) ^{\top}D^2 H (\Theta(\lambda))\big(\begin{smallmatrix} P^n-P \\ Q^n-Q \end{smallmatrix} \big) \Rightarrow^{stably}0$ by \Cref{lemmastable}(i). Consequently, \eqref{multideltaH} holds.
\end{proof}

\section{Asymptotic error distributions of symplectic methods for SHS with additive noise} \label{Sectionfour}
In this section, we consider the additive noise case and give the corresponding asymptotic error distributions of symplectic methods.
For a $2d$-dimensional SHS with additive noise
\begin{align} \label{Generaladdequ2d}
    d \begin{pmatrix} P_t \\ Q_t
\end{pmatrix}=\begin{pmatrix}
    f(P_t,Q_t) \\
    g(P_t,Q_t)
\end{pmatrix}dt +\begin{pmatrix}
    a \\ b
\end{pmatrix} dW_t, \quad t\in (0,T],
\end{align}
with initial value $(P_0,Q_0)\in \mathbb{R}^{2d}$, where $f:=-\frac{\partial H}{\partial Q}$, $g:=\frac{\partial H}{\partial P}$, and $a,b$ are constant vectors, the continuous version of the class of symplectic methods \eqref{generalsym} is
\begin{equation} \label{generalsymadd}
    \begin{split}
        P_t^n =& P_{t_k}^n + f(\theta P_t^n+(1-\theta)P_{t_k}^n, (1-\theta)Q_t^n+\theta Q_{t_k}^n)(t-t_k)  + a(W_t-W_{t_k}), \\
        Q_t^n =& Q_{t_k}^n +g(\theta P_t^n+(1-\theta)P_{t_k}^n, (1-\theta)Q_t^n+\theta Q_{t_k}^n)(t-t_k) +b(W_t-W_{t_k}),
    \end{split} \quad t\in(t_k,t_{k+1}].
\end{equation} 
\begin{assumption} \label{assumption3}
    (\romannumeral 1) $f,g:\mathbb{R}^{2d}\to \mathbb{R}^d$ are globally Lipschitz continuous.\\
    (\romannumeral 2) Let $f,g \in C^4(\mathbb{R}^{2d};\mathbb{R}^d)$ and assume that the $k$th order derivatives $D^k f,D^k g$ for $k= 2,3,4$ are functions growing at most polynomially.
\end{assumption}
Under \Cref{assumption3}, we present the asymptotic error distribution of \eqref{generalsymadd}. Here, we denote $f=(f^1,\dots,f^d),g=(g^1,\dots,g^d)$ with $f^i,g^i:\mathbb{R}^{2d}\to\mathbb{R}$ for $i=1,\dots,d$. Since the strong convergence order of \eqref{generalsymadd} in the case of additive noise is 1, the normalized error process should be defined using a normalization constant $n$.
\begin{thm} \label{mainthmadd}
    Let $U_P^n:=n(P_t^n-P_t)$ and $U_Q^n(t):=n(Q_t^n-Q_t)$, where  $(P,Q)$ is the solution to \eqref{Generaladdequ2d} and  $(P^n,Q^n)$ is defined by \eqref{generalsymadd}. 
    Then under \Cref{assumption3}, 
  we have $(U_P^n,U_Q^n)\Rightarrow^{stably} U:=(U_P,U_Q)=(U_{P,1},\dots,U_{P,d},U_{Q,1}\dots,U_{Q,d})$ in $\mathcal C([0,T];{\mathbb{R}^{2d}})$ and $U$ satisfies 
\begin{align} \label{generalUadd}
        &\quad dU_{P,i}(t) \nonumber\\
        &= \Big[\big(\frac{\partial f^i}{\partial P}\big)^{\top}(P_t,Q_t)U_P(t)+\big(\frac{\partial f^i}{\partial Q}\big)^{\top}(P_t,Q_t)U_Q(t)\Big]dt +(\theta-\frac{1}{2})\Big[\big(\frac{\partial f^i}{\partial P}\big)^{\top}f-\big(\frac{\partial f^i}{\partial Q}\big)^{\top}g\Big](P_t,Q_t)dt  \nonumber\\
        &\quad +\Big[\big(\frac{1}{2}\theta^2-\frac{1}{4}\big)a^{\top}\big(\frac{\partial^2 f^i}{\partial P^2}\big)a+\big(\theta(1-\theta)-\frac{1}{2}\big)a^{\top}\big(\frac{\partial^2 f^i}{\partial P\partial Q}\big)b  + \big(\frac{1}{2}(1-\theta)^2-\frac{1}{4}\big)b^{\top}\big(\frac{\partial^2 f^i}{\partial Q^2}\big)b \Big](P_t,Q_t)dt \nonumber \\
        &\quad +(\theta-\frac{1}{2})\Big[\big(\frac{\partial f^i}{\partial P}\big)^{\top}a-\big(\frac{\partial f^i}{\partial Q}\big)^{\top}b\Big](P_t,Q_t)dW_t-\frac{\sqrt{3}}{6}\Big[\big(\frac{\partial f^i}{\partial P}\big)^{\top}a+\big(\frac{\partial f^i}{\partial Q}\big)^{\top}b\Big](P_t,Q_t)dB_t, \nonumber\\[2mm]
        &\quad dU_{Q,i}(t) \\
        &= \Big[\big(\frac{\partial g^i}{\partial P}\big)^{\top}(P_t,Q_t)U_P(t)+\big(\frac{\partial g^i}{\partial Q}\big)^{\top}(P_t,Q_t)U_Q(t)\Big]dt +(\theta-\frac{1}{2}) \Big[\big(\frac{\partial g^i}{\partial P}\big)^{\top}f-\big(\frac{\partial g^i}{\partial Q}\big)^{\top}g\Big](P_t,Q_t)dt \nonumber\\
        &\quad +\Big[\big(\frac{1}{2}\theta^2-\frac{1}{4}\big)a^{\top}\big(\frac{\partial^2 g^i}{\partial P^2}\big)a+\big(\theta(1-\theta)-\frac{1}{2}\big)a^{\top}\big(\frac{\partial^2 g^i}{\partial P\partial Q}\big)b + \big(\frac{1}{2}(1-\theta)^2-\frac{1}{4}\big)b^{\top}\big(\frac{\partial^2 g^i}{\partial Q^2}\big)b \Big](P_t,Q_t)dt  \nonumber\\
        &\quad +(\theta-\frac{1}{2})\Big[\big(\frac{\partial g^i}{\partial P}\big)^{\top}a-\big(\frac{\partial g^i}{\partial Q}\big)^{\top}b\Big](P_t,Q_t)dW_t-\frac{\sqrt{3}}{6}\Big[\big(\frac{\partial g^i}{\partial P}\big)^{\top}a+\big(\frac{\partial g^i}{\partial Q}\big)^{\top}b\Big](P_t,Q_t)dB_t \nonumber
\end{align}
for $i=1,\dots,d$ with initial value $(U_P(0),U_Q(0))=(0,0)$, where $B$ is a Brownian motion independent of $W$.
\end{thm}

Without loss of generality, we present the proof of \Cref{mainthmadd} for the case $d=1$ and $\theta=1$.
As mentioned earlier, \eqref{generalsymadd} corresponds to the symplectic Euler method when $\theta=1$. 

Similar to the proof of \Cref{mainprop}, one can show that, for $n>2\sup|f_1'|$, the continuous numerical solution $(P^n,Q^n)$ associated with the symplectic Euler method 
is well-defined and admits an explicit expression on $\Omega$. Here, the proof is omitted.
\begin{prop} \label{mainpropadd}
    Let $\epsilon := \frac{1}{2\sup|f_1'|}$ ($\epsilon=\infty$ if $\sup|f_1'|=0$). If $n>2\sup|f_1'|$, then under Assumption \ref{assumption3}, the continuous numerical solution corresponding to the symplectic Euler method is well-defined on $\Omega$ and satisfies
    \begin{align}
    P_t^n =& \ P_{t_k}^n+f(P_{t_k}^n,Q_{t_k}^n)(t-t_k)+a(W_t-W_{t_k})+
    (f_1'f)(P_{t_k}^n,Q_{t_k}^n)(t-t_k)^2 \nonumber \\
    &+(af_1')(P_{t_k}^n,Q_{t_k}^n)(t-t_k)(W_t-W_{t_k})+\frac{1}{2}(a^2f_{11}'')(P_{t_k}^n,Q_{t_k}^n)(t-t_k)(W_t-W_{t_k})^2 \nonumber \\
    &+ (af_1'^2+af_{11}''f)(P_{t_k}^n,Q_{t_k}^n)(t-t_k)^2(W_t-W_{t_k})+(f_1'^2f+\frac{1}{2}f_{11}''f^2)(P_{t_k}^n,Q_{t_k}^n)(t-t_k)^3 \nonumber \\
    &+R_1(P_{t_k}^n,Q_{t_k}^n,t-t_k,W_t-W_{t_k}), \quad t\in (t_k,t_{k+1}], \label{explicitpadd} \\
    Q_t^n =& \ Q_{t_k}^n+g(P_{t_k}^n,Q_{t_k}^n)(t-t_k)+b(W_t-W_{t_k})+
    (g_1'f()P_{t_k}^n,Q_{t_k}^n)(t-t_k)^2 \nonumber \\
    &+(ag_1')(P_{t_k}^n,Q_{t_k}^n)(t-t_k)(W_t-W_{t_k})+\frac{1}{2}(a^2g_{11}'')(P_{t_k}^n,Q_{t_k}^n)(t-t_k)(W_t-W_{t_k})^2 \nonumber \\
    &+ (af_1'g_1'+ag_{11}''f)(P_{t_k}^n,Q_{t_k}^n)(t-t_k)^2(W_t-W_{t_k})+(f_1'g_1'f+\frac{1}{2}g_{11}''f^2)(P_{t_k}^n,Q_{t_k}^n)(t-t_k)^3 \nonumber \\
    &+R_2(P_{t_k}^n,Q_{t_k}^n,t-t_k,W_t-W_{t_k}), \quad t\in (t_k,t_{k+1}], \label{explicitqadd}
    \end{align}
    where $R_1$ and $R_2$ are real-valued continuous functions on $\mathbb{R}^2\times (-\epsilon,\epsilon)\times\mathbb{R}$ which satisfy
    \begin{align*}
    &|R_1(\zeta,h_1,h_2)|\le \sum_{\substack{\alpha+\beta=4 \\ \alpha\in\mathbb{N}^+,\beta\in\mathbb{N}}}K_1(1+\|\zeta\|^{\gamma_1})|h_1|^{\alpha}|h_2|^{\beta},  \\
    &|R_2(\zeta,h_1,h_2)|\le \sum_{\substack{\alpha+\beta=4 \\ \alpha\in\mathbb{N}^+,\beta\in\mathbb{N}}}K_2(1+\|\zeta\|^{\gamma_2})|h_1|^{\alpha}|h_2|^{\beta} 
    \end{align*}
    for some positive constants $K_1,K_2,\gamma_1,\gamma_2$.
\end{prop}
Based on \Cref{mainpropadd}, we give the integral form of $(P_t^n, Q_t^n)$. 
It follows from the It\^o formula that
\begin{align*}
    (t-t_k)(W_t-W_{t_k})^2 = 2\int_{t_k}^t(s-t_k)(W_s-W_{t_k})dW_s+\int_{t_k}^t(s-t_k)ds+\int_{t_k}^t(W_s-W_{t_k})^2ds
\end{align*}
and
\begin{align*}
    (t-t_k)^2(W_t-W_{t_k})=\int_{t_k}^t(s-t_k)^2dW_s+2\int_{t_k}^t(W_s-W_{t_k})(s-t_k)ds.
\end{align*}
Therefore, \eqref{explicitpadd} and  \eqref{explicitqadd} can be  written as 
\begin{align} 
    P_t^n =&\  P_0+\int_0^t f(P_{\eta_n(s)}^n,Q_{\eta_n(s)}^n)ds+aW_t+R_P^n(t), \label{pnadd} \\
    Q_t^n =&\  Q_0+\int_0^t g(P_{\eta_n(s)}^n,Q_{\eta_n(s)}^n)ds+bW_t+R_Q^n(t), \label{qnadd}
\end{align}
where $R_P^n:=R_{P,1}^n+R_{P,2}^n$ and $R_Q^n:=R_{Q,1}^n+R_{Q,2}^n$ with 
\begin{align*}
    R_{P,1}^n(t):=& \int_0^t(2f_1'f+\frac{1}{2}a^2f_{11}'')(P_{\eta_n(s)}^n,Q_{\eta_n(s)}^n)(s-\eta_n(s))ds + \int_0^taf_1'(P_{\eta_n(s)}^n,Q_{\eta_n(s)}^n)(s-\eta_n(s))dW_s \\ &+\int_0^taf_1'(P_{\eta_n(s)}^n,Q_{\eta_n(s)}^n)(W_s-W_{\eta_n(s)})ds+ \frac{1}{2}\int_0^ta^2f_{11}''(P_{\eta_n(s)}^n,Q_{\eta_n(s)}^n)(W_s-W_{\eta_n(s)})^2ds, \\
    R_{P,2}^n(t):=& \int_0^ta^2f_{11}''(P_{\eta_n(s)}^n,Q_{\eta_n(s)}^n)(s-\eta_n(s))(W_s-W_{\eta_n(s)})dW_s\\
    &+\int_0^t (af_1'^2+af_{11}''f)(P_{\eta_n(s)}^n,Q_{\eta_n(s)}^n)(s-\eta_n(s))^2dW_s \\
    &+2\int_0^t(af_1'^2+af_{11}''f)(P_{\eta_n(s)}^n,Q_{\eta_n(s)}^n)(W_s-W_{\eta_n(s)})(s-\eta_n(s))ds \\
    &+3\int_0^t(f_1'^2f+\frac{1}{2}f_{11}''f^2)(P_{\eta_n(s)}^n,Q_{\eta_n(s)}^n)(s-\eta_n(s))^2ds \\
    &+\sum_{k=0}^{[nt]-1}R_1(P_{t_k}^n,Q_{t_k}^n,t_{k+1}-t_k,W_{t_{k+1}}-W_{t_k}) +R_1(P_{\eta_n(t)}^n,Q_{\eta_n(t)}^n,t-\eta_n(t),W_t-W_{\eta_n(t)}),\\
    R_{Q,1}^n(t):=&\int_0^t(2g_1'f+\frac{1}{2}a^2g_{11}'')(P_{\eta_n(s)}^n,Q_{\eta_n(s)}^n)(s-\eta_n(s))ds + \int_0^tag_1'(P_{\eta_n(s)}^n,Q_{\eta_n(s)}^n)(s-\eta_n(s))dW_s \\ &+\int_0^tag_1'(P_{\eta_n(s)}^n,Q_{\eta_n(s)}^n)(W_s-W_{\eta_n(s)})ds+ \frac{1}{2}\int_0^ta^2g_{11}''(P_{\eta_n(s)}^n,Q_{\eta_n(s)}^n)(W_s-W_{\eta_n(s)})^2ds, \\
    R_{Q,2}^n(t):=& \int_0^ta^2g_{11}''(P_{\eta_n(s)}^n,Q_{\eta_n(s)}^n)(s-\eta_n(s))(W_s-W_{\eta_n(s)})dW_s\\
    &+\int_0^t (af_1'g_1'+ag_{11}''f)(P_{\eta_n(s)}^n,Q_{\eta_n(s)}^n)(s-\eta_n(s))^2dW_s \\
    &+2\int_0^t(af_1'g_1'+ag_{11}''f)(P_{\eta_n(s)}^n,Q_{\eta_n(s)}^n)(W_s-W_{\eta_n(s)})(s-\eta_n(s))ds \\
    &+3\int_0^t(f_1'g_1'f+\frac{1}{2}g_{11}''f^2)(P_{\eta_n(s)}^n,Q_{\eta_n(s)}^n)(s-\eta_n(s))^2ds \\
    &+\sum_{k=0}^{[nt]-1}R_2(P_{t_k}^n,Q_{t_k}^n,t_{k+1}-t_k,W_{t_{k+1}}-W_{t_k}) +R_2(P_{\eta_n(t)}^n,Q_{\eta_n(t)}^n,t-\eta_n(t),W_t-W_{\eta_n(t)}).
\end{align*}

We state the convergence results for the remainder terms in the following proposition. The proof is analogous to that of \Cref{propconverge0} and is therefore omitted.
\begin{prop} \label{propconverge0add}
    Let \Cref{assumption3} hold, then $R_P^n,R_Q^n \to^{\mathbb{P}} 0$ and $nR_{P,2}^n,nR_{Q,2}^n\to^{\mathbb{P}}0$ in $\mathcal{C}([0,T];\mathbb{R})$ as $n\to \infty$.
\end{prop}
\begin{proof}[Proof of \Cref{mainthmadd}]
Given \eqref{pnadd}, \eqref{qnadd} and the fact that $R_P^n, R_Q^n \to^{\mathbb{P}}0$, \Cref{propmain} yields $(P^n, Q^n, W) \to^{\mathbb{P}} (P, Q, W)$ as $n \to \infty$. Consequently, we also have $(P^n, Q^n, W) \Rightarrow (P, Q, W)$.
For $U_P^n = n(P^n-P)$, comparing \eqref{Generaladdequ2d} and \eqref{pnadd} leads to
\begin{align} \label{equ_add_expressionUpn}
    U_P^n(t) = n\int_0^t[f(P_{\eta_n(s)}^n,Q_{\eta_n(s)}^n)-f(P_s,Q_s)]ds+nR_{P,1}^n(t)+nR_{P,2}^n(t),
\end{align}
where $nR_{P,2}^n\to^\mathbb{P}0$ (see \Cref{propconverge0add}). 
Note that
\begin{align} \label{nsds}
    \lim_{n\to \infty} n\int_0^t(s-\eta_n(s))ds=\frac{1}{2}t,
\end{align}
and by \cite[Theorem 1.2]{Rootzen1980}, we have
\begin{equation} \label{nsdW}
    \begin{split} 
        n\int_0^{\cdot}(s-\eta_n(s))dW_s &= n\int_0^{\cdot}(s+\frac{1}{2n}-\frac{1}{2n}-\eta_n(s))dW_s \\
        &=\frac{1}{2}W+n\int_0^{\cdot}(s-\frac{1}{2n}-\eta_n(s))dW_s \Rightarrow \frac{1}{2}W-\frac{\sqrt{3}}{6}B,
    \end{split}
\end{equation}
and 
\begin{equation} \label{nWds}
    \begin{split} 
        n\int_0^{\cdot}(W_s-W_{\eta_n(s)})ds &= n\int_0^{\cdot}\int_{\eta_n(s)}^sdW_uds = n\int_0^{\cdot}\int_u^{\eta_n(u)+1/n}dsdW_u \\
        &=n\int_0^{\cdot}(\eta_n(u)+\frac{1}{n}-u)dW_u \Rightarrow \frac{1}{2}W+\frac{\sqrt{3}}{6}B,
    \end{split}
\end{equation}
where $B$ is a Brownian motion independent of $W$. Moreover, we arrive at 
\begin{equation} \label{nW^2ds}
    \begin{split}
        n\int_0^{\cdot} (W_s-W_{\eta_n(s)})^2ds &= 2n\int_0^{\cdot} \int_{\eta_n(s)}^s(W_u-W_{\eta_n(u)})dW_uds+n\int_0^{\cdot}(s-\eta_n(s))ds \Rightarrow \frac{1}{2}\int^{\cdot}_0 ds,
    \end{split}
\end{equation}
which follows from the fact that
\begin{align*}
    &2n\int_0^{\cdot} \int_{\eta_n(s)}^s(W_u-W_{\eta_n(u)})dW_uds\\
    =&\ 2n\int_0^{\cdot}\int_u^{\eta_n(u)+1/n}(W_u-W_{\eta_n(u)})dsdW_u 
    =2n\int_0^{\cdot}(W_u-W_{\eta_n(u)})(\eta_n(u)+\frac{1}{n}-u)dW_u \\
    =&\ 2\int_0^{\cdot}(W_u-W_{\eta_n(u)})dW_u-2n\int_0^{\cdot}(W_u-W_{\eta_n(u)})(u-\eta_n(u))dW_u \to^\mathbb{P} 0
\end{align*}
in $\mathcal{C}([0,T];\mathbb{R})$.
The goodness of the integrals on the left-hand side of equations \eqref{nsds}-\eqref{nW^2ds} is justified by \cite[Theorem 2.7]{Burdzy2010}.
Given that $(P^n,Q^n)\Rightarrow(P,Q)$, it follows from \eqref{nsds}-\eqref{nW^2ds} and \cite[Lemma 3.2]{KurtzProtter1991b} that 
\begin{align*}
    nR_{P,1}^n \Rightarrow \int_0^{\cdot}(f_1'f+\frac{1}{2}a^2f_{11}'')(P_s,Q_s)ds+\int_0^{\cdot}(af_1')(P_s,Q_s)dW_s
\end{align*}
in $\mathcal C([0,T];\mathbb R)$.
Next, it remains to analyze the convergence of the first term on the right-hand side of \eqref{equ_add_expressionUpn}, which can be decomposed as
\begin{align*}
    &n\int_0^t[f(P_{\eta_n(s)}^n,Q_{\eta_n(s)}^n)-f(P_s,Q_s)]ds \\
    =& \ n\int_0^t[f(P_s^n,Q_s^n)-f(P_s,Q_s)]ds -n\int_0^t[f(P_s^n,Q_s^n)-f(P_{\eta_n(s)}^n,Q_{\eta_n(s)}^n)]ds
    =: S_1^n(t)-S_2^n(t).
\end{align*}
For $S_1^n$,
\begin{align*}
    S_1^n(t) = \int_0^t[f_1'(P_s,Q_s)U_P^n(s)+f_2'(P_s,Q_s)U_Q^n(s)]ds+S_{1_R}^n(t),
\end{align*}
where
\begin{align*}
    S_{1_R}^n(t):=&\  n \int_0^t \int_0^1 (1-\lambda)(X_s^n-X_s)^{\top}D^2f(\Theta_1(\lambda,s))(X_s^n-X_s)d\lambda ds.
\end{align*}
with $X^n:=(P^n,Q^n)^{\top}$ and $X:=(P,Q)^{\top}$. Thus, an analogous argument to that used for the first term of $I_{1_R}^n$ in \eqref{I1Rn} shows that $S_{1_R}^n\Rightarrow0$.
For $S_2^n$, plugging \eqref{explicitpadd} and \eqref{explicitqadd} into $S_2^n$ gives that
\begin{align*}
    S_2^n(t) =& \ n\int_0^tf_1'(P_{\eta_n(s)}^n,Q_{\eta_n(s)}^n)(P_s^n-P_{\eta_n(s)}^n)ds+n\int_0^tf_2'(P_{\eta_n(s)}^n,Q_{\eta_n(s)}^n)(Q_s^n-Q_{\eta_n(s)}^n)ds \\
    &+\frac{1}{2}n\int_0^t f_{11}''(P_{\eta_n(s)}^n,Q_{\eta_n(s)}^n)(P_s^n-P_{\eta_n(s)}^n)^2ds+\frac{1}{2}n\int_0^t f_{22}''(P_{\eta_n(s)}^n,Q_{\eta_n(s)}^n)(Q_s^n-Q_{\eta_n(s)}^n)^2ds \\
    &+n\int_0^t f_{12}''(P_{\eta_n(s)}^n,Q_{\eta_n(s)}^n)(P_s^n-P_{\eta_n(s)}^n)(Q_s^n-Q_{\eta_n(s)}^n)ds + S_{2_R}^n(t) \\
    =&\  n\int_0^t(f_1'f+f_2'g)(P_{\eta_n(s)}^n,Q_{\eta_n(s)}^n)(s-\eta_n(s))ds +n\int_0^t(af_1'+bf_2')(P_{\eta_n(s)}^n,Q_{\eta_n(s)}^n)(W_s-W_{\eta_n(s)})ds \\
    &+\frac{1}{2}n\int_0^t(a^2f_{11}''+b^2f_{22}''+2abf_{12}'')(P_{\eta_n(s)}^n,Q_{\eta_n(s)}^n)(W_s-W_{\eta_n(s)})^2ds+S_{21}^n(t)+S_{2_R}^n(t),
\end{align*}
where $S_{21}^n(t)$ consists of higher-order terms such as $n\int_0^t af_1'^2(P_{\eta_n(s)}^n,Q_{\eta_n(s)}^n)(s-\eta_n(s))(W_s-W_{\eta_n(s)})ds$. By a similar argument to the proof of \Cref{propconverge0}, it is clear that $S_{21}^n\Rightarrow 0$. $S_{2_R}^n$ denotes the remainder in the integral form of the Taylor expansion, that is,
\begin{align*}
    S_{2_R}^n(t) = \frac{n}{2}\int_0^t\int_0^1 (1-\lambda)^2\Big( (P_s^n-P_{\eta_n(s)}^n)\frac{\partial}{\partial P}+(Q_s^n-Q_{\eta_n(s)}^n)\frac{\partial}{\partial Q}\Big)^3f(\Theta_2(\lambda,s))d\lambda ds.
\end{align*}
As in the convergence analysis of the second term of $I_{1_R}^n$ in \eqref{I1Rn}, it holds that
\begin{align*}
    &\mathbb E\Big[\sup_{t\in[0,T]}|S_{2_R}^n(t)|^2\Big] \\
    \le &\ n^2C\mathbb E\Big[ \int_0^T \big(\|X^n_s-X_{\eta_n(s)}^n\|^6+\|X_{\eta_n(s)}^n\|^{2\gamma}\|X^n_s-X_{\eta_n(s)}^n\|^6+\|X^n_s-X_{\eta_n(s)}^n\|^{2\gamma+6}\big)ds  \Big]\to 0.
\end{align*} 
Thus, by \eqref{nsds}-\eqref{nW^2ds},
\begin{align*}
    S_2^n \Rightarrow &\ \frac{1}{2}\int_0^{\cdot} (f_1'f+f_2'g)(P_s,Q_s)ds+\frac{1}{4}\int_0^{\cdot}(a^2f_{11}''+b^2f_{22}''+2abf_{12}'')(P_s,Q_s)ds \\
    &+\frac{1}{2}\int_0^{\cdot}(af_1'+bf_2')(P_s,Q_s)dW_s+\frac{\sqrt{3}}{6}\int_0^{\cdot}(af_1'+bf_2')(P_s,Q_s)dB_s.
\end{align*}

Collecting the estimates above yields 
\begin{align*}
    U_P^n(t) = \int_0^t[f_1'(P_s,Q_s)U_P^n(s)+f_2'(P_s,Q_s)U_Q^n(s)]ds+M_1^n(t),
\end{align*}
where $M_1^n:=S_{1_R}^n-S_2^n+nR_{P,1}^n+nR_{P,2}^n$ satisfies
\begin{align*}
    M_1^n\Rightarrow &\ \frac{1}{2}\int_0^{\cdot} (f_1'f-f_2'g)(P_s,Q_s)ds+\frac{1}{4}\int_0^{\cdot}(a^2f_{11}''-b^2f_{22}''-2abf_{12}'')(P_s,Q_s)ds \\
    &+\frac{1}{2}\int_0^{\cdot}(af_1'-bf_2')(P_s,Q_s)dW_s-\frac{\sqrt{3}}{6}\int_0^{\cdot}(af_1'+bf_2')(P_s,Q_s)dB_s
\end{align*}
in $\mathcal C([0,T];\mathbb R)$.
Applying a similar analysis to $U_Q^n$, we have 
\begin{align*}
    U_Q^n(t) = \int_0^t[g_1'(P_s,Q_s)U_P^n(s)+g_2'(P_s,Q_s)U_Q^n(s)]ds+M_2^n(t),
\end{align*}
where
\begin{align*}
    M_2^n \Rightarrow &\ \frac{1}{2}\int_0^{\cdot} (g_1'f-g_2'g)(P_s,Q_s)ds+\frac{1}{4}\int_0^{\cdot}(a^2g_{11}''-b^2g_{22}''-2abg_{12}'')(P_s,Q_s)ds \\
    &+\frac{1}{2}\int_0^{\cdot}(ag_1'-bg_2')(P_s,Q_s)dW_s-\frac{\sqrt{3}}{6}\int_0^{\cdot}(ag_1'+bg_2')(P_s,Q_s)dB_s.
\end{align*}
Finally, by \Cref{rmk1}, we obtain $(U_P^n,U_Q^n,W) \Rightarrow (U_P,U_Q,W)$ and $(U_P^n,U_Q^n)$ satisfies \eqref{generalUadd} with $d=1$ and $\theta=1$. Furthermore, by \Cref{lemmastable}(\romannumeral 2), $(U_P^n,U_Q^n)\Rightarrow^{stably} U:=(U_P,U_Q)$.
\end{proof}

Similarly, the obtained asymptotic error distribution in the additive noise case also admits a stochastic Hamiltonian formulation, and the limiting distribution of the normalized Hamiltonian deviation can be established as well. The proofs follow the same arguments as those in \Cref{3.3} and are therefore omitted.
\begin{thm} \label{UisHamiltonianadd}
    The asymptotic error distribution  $U:=(U_P,U_Q)$ given in \eqref{generalUadd} still
has a Hamiltonian formulation, satisfying
    \begin{equation*} 
    \begin{split}
        d \begin{pmatrix}
            U_P \\
            U_Q
        \end{pmatrix} =& \begin{pmatrix}
            -\frac{\partial H_0}{\partial U_Q} \\
            \frac{\partial H_0}{\partial U_P}
        \end{pmatrix} dt+\begin{pmatrix}
            -\frac{\partial H_1}{\partial U_Q} \\
            \frac{\partial H_0}{\partial U_P}
        \end{pmatrix} \circ dW_t + \begin{pmatrix}
            -\frac{\partial H_2}{\partial U_Q} \\
            \frac{\partial H_0}{\partial U_P}
        \end{pmatrix} \circ dB_t,
    \end{split}
    \end{equation*}
    where 
    \begin{align*}
        H_0 =& \ \frac{1}{2}U_P^{\top}\big(\frac{\partial g}{\partial P}\big)(P,Q) U_P-U_P^{\top}\big(\frac{\partial f}{\partial P}\big)(P,Q)U_Q-\frac{1}{2}U_Q^{\top}\big(\frac{\partial f}{\partial Q}\big)(P,Q)U_Q \\
        &+(\theta-\frac{1}{2})\sum_{i=1}^d \bigg(\Big[\big(\frac{\partial g^i}{\partial P}\big)^{\top}f-\big(\frac{\partial g^i}{\partial Q}\big)^{\top}g\Big](P,Q)U_{P,i}-\Big[\big(\frac{\partial f^i}{\partial P}\big)^{\top}f-\big(\frac{\partial f^i}{\partial Q}\big)^{\top}g\Big](P,Q)U_{Q,i}\bigg) \\
        &-\sum_{i=1}^d \Big(\frac{1}{2}\theta(1-\theta)a^{\top}\frac{\partial^2 g^i}{\partial P^2}a-\big(\theta(1-\theta)-\frac{1}{2}\big)a^{\top}\frac{\partial^2 g^i}{\partial P\partial Q}b+\frac{1}{2}\theta(1-\theta)b^{\top}\frac{\partial^2 g^i}{\partial Q^2}b\Big)(P,Q)U_{P,i} \\
        &+\sum_{i=1}^d \Big(\frac{1}{2}\theta(1-\theta)a^{\top}\frac{\partial^2 f^i}{\partial P^2}a-\big(\theta(1-\theta)-\frac{1}{2}\big)a^{\top}\frac{\partial^2 f^i}{\partial P\partial Q}b+\frac{1}{2}\theta(1-\theta)b^{\top}\frac{\partial^2 f^i}{\partial Q^2}b\Big)(P,Q)U_{Q,i}, \\
        H_1 =&\  (\theta-\frac{1}{2})\sum_{i=1}^d \bigg(\Big[\big(\frac{\partial g^i}{\partial P}\big)^{\top}a-\big(\frac{\partial g^i}{\partial Q}\big)^{\top}b\Big](P,Q)U_{P,i}-\Big[\big(\frac{\partial f^i}{\partial P}\big)^{\top}a-\big(\frac{\partial f^i}{\partial Q}\big)^{\top}b\Big](P,Q)U_{Q,i}\bigg), \\
        H_2 =&\  -\frac{\sqrt{3}}{6} \sum_{i=1}^d \bigg(\Big[\big(\frac{\partial g^i}{\partial P}\big)^{\top}a+\big(\frac{\partial g^i}{\partial Q}\big)^{\top}b\Big](P,Q)U_{P,i}-\Big[\big(\frac{\partial f^i}{\partial P}\big)^{\top}a+\big(\frac{\partial f^i}{\partial Q}\big)^{\top}b\Big](P,Q)U_{Q,i}\bigg).
    \end{align*}
\end{thm}

\begin{thm} \label{thmdeltaHadd}
    If $H$ is one of the Hamiltonians in \eqref{Generaladdequ2d} and $U=(U_P,U_Q)$ is the asymptotic error distribution given in \eqref{generalUadd}, then 
    \begin{align*}
        n(H(P^n,Q^n)-H(P,Q))\Rightarrow^{stably} \Big(\frac{\partial H}{\partial P} (P,Q) \Big)^{\top} U_P+\Big(\frac{\partial H}{\partial Q}(P,Q)\Big)^{\top} U_Q
    \end{align*}
    in $\mathcal C([0,T];{\mathbb{R}})$.
\end{thm}

\section{A new approach for deriving the asymptotic error distribution via stochastic modified equation} \label{Sectionfive}
In this section, we introduce the stochastic modified equation in the It\^o sense with respect to strong convergence (see \cite{Deng2016} in the Stratonovich sense) and propose a new approach for deriving the asymptotic error distribution. 

\subsection{Construction of stochastic modified equation}

For a stochastic differential equation
\begin{align} \label{SDE}
    dX_t = \sum_{r=0}^m f_r(X_t)dW_t^r,
\end{align}
where $X_t \in \mathbb{R}^{\bf d}$, $f_r:\mathbb{R}^{\bf d}\to \mathbb{R}^{\bf d}$, $W^r,r=1,\dots,m$ are independent 1-dimensional Brownian motions, and $t$ is denoted by $W^0$ for notational convenience. To express the stochastic modified equation, we first define the multiple It\^o integral 
\begin{align*}
    I_{\alpha,t} = \int_0^t\int_0^{s_l} \dots \int_0^{s_2} dW_{s_1}^{j_1}\dots dW_{s_{l-1}}^{j_{l-1}}dW_{s_l}^{j_l},
\end{align*}
where $\alpha=(j_1,\dots,j_l)$ with  $j_i\in \{0,1,\dots,m\}$ for $i=1,\dots,l$ is a multi-index. The length of $\alpha$ is denoted by $l(\alpha)$. A multi-index of length zero $v$ is included for completeness with $I_{v,t}=1$.
We define the stochastic processes $Y_{\alpha,t}$ for the step size $\frac 1n$ by $Y_{v,t}:=1$ and 
\begin{align*}
    Y_{\alpha,t} := \int_{\frac kn}^t\int_{\frac kn}^{s_l} \dots \int_{\frac kn}^{s_2} dW_{s_1}^{j_1}\dots dW_{s_{l-1}}^{j_{l-1}}dW_{s_l}^{j_l}, \quad \frac kn<t\le \frac {k+1}{n}, \quad k\in\mathbb{N}
\end{align*}
for $l(\alpha)>0$. It is clear that $Y_{\alpha,t}=I_{\alpha,t}$ for $0<t\le \frac 1n$.


We suppose that the stochastic modified equation with respect to strong convergence can be expressed as
\begin{align} \label{generalME}
    \tilde{X_t} = X_0+\sum_{r=0}^m \int_0^t\sum_\alpha \tilde{f}_{r,\alpha}(\tilde{X_s})Y_{\alpha,s} dW_s^r
\end{align}
with
$\tilde{f}_{r,\alpha}:\mathbb{R}^{\bf d} \to \mathbb{R}^{\bf d}$ and further denote
$\tilde{X}_t=(\tilde{X}_{t,1},\dots,\tilde{X}_{t,\bf d})$ and $\tilde{f}_{r,\alpha}=(\tilde{f}_{r,\alpha,1},\dots,\tilde{f}_{r,\alpha,\bf d})$ with $\tilde{f}_{r,\alpha,i}:\mathbb{R}^{\bf d} \to \mathbb{R}$ for $i=1,\dots,\bf d$. 
When $0<t\le \frac 1n$, it follows from the It\^o formula that
\begin{align*}
    d(\tilde{f}_{r,\alpha,i}(\tilde{X}_s)) =&\  (\nabla \tilde{f}_{r,\alpha,i}(\tilde{X}_s))^{\top}\sum_{k=0}^m\sum_{\beta}\tilde{f}_{k,\beta}(\tilde{X}_s)I_{\beta,s}dW_s^k + \frac{1}{2}(d\tilde{X}_s)^{\top} \nabla^2\tilde{f}_{r,\alpha,i}(\tilde{X}_s)(d\tilde{X}_s) \\
    =& \sum_{j=1}^d(\tilde{f}_{r,\alpha,i})_{x_j}(\tilde{X}_s)\sum_{k=0}^m\sum_{\beta}\tilde{f}_{k,\beta,j}(\tilde{X}_s)I_{\beta,s}dW_s^k \\
    &+ \frac{1}{2}\sum_{j,l=1}^d(\tilde{f}_{r,\alpha,i})_{x_jx_l}(\tilde{X}_s)\sum_{k=1}^m\sum_\beta \sum_\gamma\tilde{f}_{k,\beta,j}\tilde{f}_{k,\gamma,l}(\tilde{X}_s)I_{\beta,s}I_{\gamma,s}ds,
\end{align*}
where $(\tilde{f}_{r,\alpha,i})_{x_j} = \frac{\partial \tilde{f}_{r,\alpha,i}}{\partial x_j}$, $(\tilde{f}_{r,\alpha,i})_{x_jx_l} = \frac{\partial^2 \tilde{f}_{r,\alpha,i}}{\partial x_j \partial x_l}$. This leads to
\begin{align} \label{ME}
        \tilde{X}_{t,i} =&\  X_{0,i} + \sum_{r=0}^m \sum_\alpha \int_0^t \tilde{f}_{r,\alpha,i}(X_0)I_{\alpha,s}dW_s^r \nonumber \\
        &+ \sum_{r=0}^m \sum_\alpha \sum_{j=1}^d \sum_{k=0}^m\sum_\beta \int_0^t[\int_0^s(\tilde{f}_{r,\alpha,i})_{x_j}\tilde{f}_{k,\beta,j}(\tilde{X}_\tau)I_{\beta,\tau}dW_{\tau}^k]I_{\alpha,s}dW_s^r \\
        &+\sum_{r=0}^m\sum_\alpha \sum_{j,l=1}^d\sum_{k=1}^m\sum_\beta \sum_\gamma \frac{1}{2} \int_0^t[\int_0^s(\tilde{f}_{r,\alpha,i})_{x_jx_l}\tilde{f}_{k,\beta,j}\tilde{f}_{k,\gamma,l}(\tilde{X}_\tau)I_{\beta,\tau}I_{\gamma,\tau}d\tau]I_{\alpha,s}dW_s^r \nonumber
\end{align}
for $0<t\leq \frac 1n$.
Furthermore, the terms involving $\tilde{X}_{\tau}$  on the right-hand side of \eqref{ME} can also be expanded at $X_0$ using the It\^o formula. Then, the stochastic modified equation is obtained by matching its coefficients with those of the numerical method. 

Without loss of generality, we first give the truncated stochastic modified equation of the symplectic Euler method for 
\eqref{Generalequ} (i.e. $m=1$ and ${\bf d}=2$), from which the general result follows by analogous arguments. 
Comparing \eqref{ME} with \eqref{pn} and \eqref{qn} on $0<t\le \frac 1n$ yields $\tilde{f}_{0,v,1}=f+\frac{1}{2}a_1'a+\frac{1}{2}a_2'b$, $\tilde{f}_{0,v,2}=g+\frac{1}{2}b_1'a+\frac{1}{2}b_2'b$, $\tilde{f}_{1,v,1}=a$, $\tilde{f}_{1,v,2}=b$.
Note that for $\tilde{X}_{t,i}$ with $i=1,2$ and $0<t\le \frac 1n$, the coefficient of $\int_0^tW_sdW_s$ in the expansion form of \eqref{ME} at $X_0$ is 
\begin{align*}
    \tilde{f}_{1,(1),i}(X_0)+\sum_{j=1}^d (\tilde{f}_{1,v,i})_{x_j}\tilde{f}_{1,v,j}(X_0).
\end{align*}
Thus, it follows that
$$\tilde{f}_{1,(1),1}+(\tilde{f}_{1,v,1})_1'\tilde{f}_{1,v,1}+(\tilde{f}_{1,v,1})_2'\tilde{f}_{1,v,2}=2a_1'a,$$ 
$$\tilde{f}_{1,(1),2}+(\tilde{f}_{1,v,2})_1'\tilde{f}_{1,v,1}+(\tilde{f}_{1,v,2})_2'\tilde{f}_{1,v,2}=2b_1'a,$$ which further show that 
$\tilde{f}_{1,(1),1} = a_1'a-a_2'b$ and $\tilde{f}_{1,(1),2} = b_1'a-b_2'b$.
Plugging these coefficients into \eqref{generalME}, we have the truncated stochastic modified equation as follows 
\begin{align} \label{MEofsym}
        \tilde{P}_t^n =& \ P_0+\int_0^t (f+\frac{1}{2}a_1'a+\frac{1}{2}a_2'b)(\tilde{P}_s^n,\tilde{Q}_s^n)ds+\int_0^ta(\tilde{P}_s^n,\tilde{Q}_s^n)dW_s \nonumber \\
        &+ \int_0^t (a_1'a-a_2'b)(\tilde{P}_s^n,\tilde{Q}_s^n)(W_s-W_{\eta_n(s)})dW_s, \nonumber \\
        \tilde{Q}_t^n =& \ Q_0+\int_0^t (g+\frac{1}{2}b_1'a+\frac{1}{2}b_2'b)(\tilde{P}_s^n,\tilde{Q}_s^n)ds+\int_0^t b(\tilde{P}_s^n,\tilde{Q}_s^n)dW_s \nonumber\\
        &+\int_0^t (b_1'a-b_2'b)(\tilde{P}_s^n,\tilde{Q}_s^n)(W_s-W_{\eta_n(s)})dW_s.
\end{align}

Next, we give the truncated stochastic modified equation for the additive noise case. For simplicity, we also take $m=1$ and ${\bf d}=2$ and give the result for the symplectic Euler method.
Using the It\^o formula, we rewrite both \eqref{pnadd} and \eqref{ME} into the representation consisting of multiple It\^o integrals for $0<t\le \frac 1n$. 
Comparing coefficients yields 
$\tilde{f}_{0,v,1}=f$, $\tilde{f}_{0,v,2}=g$, $\tilde{f}_{1,v,1}=a$, $\tilde{f}_{1,v,2}=b$. 
The coefficients of the following multiple integrals are
\begin{align*}
    \int_0^t sds: \quad & \tilde{f}_{0,(0),i}(X_0)+\sum_{j=1}^2 (\tilde{f}_{0,v,i})_{x_j}\tilde{f}_{0,v,j}(X_0)+ \sum_{j=1}^2 (\tilde{f}_{0,(1),i})_{x_j}\tilde{f}_{1,v,j}  \\
    &+\sum_{j,l=1}^2\frac{1}{2}(\tilde{f}_{0,v,i})_{x_jx_l}\tilde{f}_{1,v,j}\tilde{f}_{1,v,l}(X_0),\\
    \int_0^tW_sds:\quad & \tilde{f}_{0,(1),i}(X_0)+\sum_{j=1}^2(\tilde{f}_{0,v,i})_{x_j}\tilde{f}_{1,v,j}(X_0),\\
    \int_0^t sdW_s: \quad & \tilde{f}_{1,(0),i}(X_0), \\
    \int_0^tW_sdW_s: \quad & \tilde{f}_{1,(1),i}(X_0), 
\end{align*}
where we use the fact that $a$ and $b$ are constants.
The term $\sum_{j=1}^2 (\tilde{f}_{0,(1),i})_{x_j}\tilde{f}_{1,v,j}$ in the coefficient of $\int_0^tsds$ arises from the second term on the right-hand side of \eqref{ME}, due to the identity $\int_0^t W_s^2ds = 2\int_0^t\int_0^sW_{\tau}dW_{\tau}ds+\int_0^tsds$.
Thus, we have
$\tilde{f}_{1,(1),1}=0$, $\tilde{f}_{1,(1),2}=0$, $\tilde{f}_{1,(0),1}=af_1'$, $\tilde{f}_{1,(0),2}=ag_1'$, $\tilde{f}_{0,(1),1}=-bf_2'$, $\tilde{f}_{0,(1),2}=-bg_2'$, and
\begin{align*}
    & \tilde{f}_{0,(0),1}=f_1'f-f_2'g+\frac{1}{2}a^2f_{11}''+\frac{1}{2}b^2f_{22}'', \\
    & \tilde{f}_{0,(0),2} = g_1'f-g_2'g+\frac{1}{2}a^2g_{11}''+\frac{1}{2}b^2g_{22}''.
\end{align*}
Hence, we obtain the truncated stochastic modified equation: 
\begin{align}  \label{MEofsymadd}
        \tilde{P}_t^n =& \ P_0+ \int_0^tf(\tilde{P}_s^n,\tilde{Q}_s^n)ds+aW_t +\int_0^t(f_1'f-f_2'g+\frac{1}{2}a^2f_{11}''+\frac{1}{2}b^2f_{22}'')(\tilde{P}_s^n,\tilde{Q}_s^n)(s-\eta_n(s))ds \nonumber \\
        &+\int_0^taf_1'(\tilde{P}_s^n,\tilde{Q}_s^n)(s-\eta_n(s))dW_s -\int_0^t bf_2'(\tilde{P}_s^n,\tilde{Q}_s^n)(W_s-W_{\eta_n(s)})ds, \nonumber \\
        \tilde{Q}_t^n =& \ Q_0+ \int_0^tg(\tilde{P}_s^n,\tilde{Q}_s^n)ds+bW_t +\int_0^t(g_1'f-g_2'g+\frac{1}{2}a^2g_{11}''+\frac{1}{2}b^2g_{22}'')(\tilde{P}_s^n,\tilde{Q}_s^n)(s-\eta_n(s))ds \nonumber \\
        &+\int_0^tag_1'(\tilde{P}_s^n,\tilde{Q}_s^n)(s-\eta_n(s))dW_s -\int_0^t bg_2'(\tilde{P}_s^n,\tilde{Q}_s^n)(W_s-W_{\eta_n(s)})ds.
\end{align}


\subsection{A new approach for deriving the asymptotic error distribution}
In this subsection, we give the asymptotic error distribution of the truncated stochastic modified equation, 
which is the same as that we obtained for numerical methods in previous sections. 
\begin{thm} \label{modifiedthm}
    Let $\tilde{U}_P^n(t):=\sqrt{n}(\tilde{P}_t^n-P_t)$ and $\tilde{U}_Q^n(t):=\sqrt{n}(\tilde{Q}_t^n-Q_t)$, 
    where  $(P,Q)$ is the solution to \eqref{Generalequ} and $(\tilde{P}^n,\tilde{Q}^n)$ is the truncated strong modified equation of the symplectic Euler method given in \eqref{MEofsym}. Then under \Cref{assumption1}, we have $(\tilde{U}_P^n,\tilde{U}_Q^n)\Rightarrow^{stably} U=(U_P,U_Q)$ in $\mathcal{C}([0,T];\mathbb{R}^2)$ and $U$ satisfies  \eqref{symEulerU}.
\end{thm}
\begin{proof}
By Remark \ref{rmk1}, comparing \eqref{Generalequ} and \eqref{MEofsym} leads to $(\tilde{P}^n,\tilde{Q}^n)\to^{\mathbb{P}}(P,Q)$ as $n\to \infty$. Moreover,
\begin{align*}
    \tilde{U}_P^n(t) =& \ \sqrt{n}\int_0^t [(f+\frac{1}{2}a_1'a+\frac{1}{2}a_2'b)(\tilde{P}_s^n,\tilde{Q}_s^n)-(f+\frac{1}{2}a_1'a+\frac{1}{2}a_2'b)(P_s,Q_s)]ds\\
    &+\sqrt{n}\int_0^t[a(\tilde{P}_s^n,\tilde{Q}_s^n)-a(P_s,Q_s)]dW_s +\sqrt{n} \int_0^t (a_1'a-a_2'b)(\tilde{P}_s^n,\tilde{Q}_s^n)(W_s-W_{\eta_n(s)})dW_s \\
    =& \int_0^t [(f+\frac{1}{2}a_1'a+\frac{1}{2}a_2'b)_1'(P_s,Q_s)\tilde{U}_P^n(s)+(f+\frac{1}{2}a_1'a+\frac{1}{2}a_2'b)_2'(P_s,Q_s)\tilde{U}_Q^n(s)]ds \\ &+\int_0^t[a_1'(P_s,Q_s)\tilde{U}_P^n(s)+a_2'(P_s,Q_s)\tilde{U}_Q^n(s)]dW_s+ \tilde{T}_1^n(t)
\end{align*} 
where $\tilde{T}_1^n \Rightarrow \frac{1}{\sqrt{2}}\int_0^{\cdot} (a_1'a-a_2'b)(P_s,Q_s)dB_s$ as $n\to \infty$, which relies on the fact that
\begin{align*}
    \sup_{t\in[0,T]}\big(\mathbb{E}[\|\tilde{X}_t^n-X_t\|^{2q}]\big)^{1/2q} \le C(\frac{1}{n})^{1/2}
\end{align*}  
for any $q\ge1$. Similarly, we have 
\begin{align*}
    \tilde{U}_Q^n(t) =&\int_0^t [(g+\frac{1}{2}b_1'a+\frac{1}{2}b_2'b)_1'(P_s,Q_s)\tilde{U}_P^n(s)+(g+\frac{1}{2}b_1'a+\frac{1}{2}b_2'b)_2'(P_s,Q_s)\tilde{U}_Q^n(s)]ds \\ &+\int_0^t[b_1'(P_s,Q_s)\tilde{U}_P^n(s)+b_2'(P_s,Q_s)\tilde{U}_Q^n(s)]dW_s+ \tilde{T}_2^n(t)
\end{align*}
where $\tilde{T}_2^n \Rightarrow \frac{1}{\sqrt{2}}\int_0^{\cdot} (b_1'a-b_2'b)(P_s,Q_s)dB_s$ as $n\to \infty$.
Thus, by \Cref{lemmastable} and \Cref{rmk1}, we conclude that $(\tilde{U}_P^n,\tilde{U}_Q^n)\Rightarrow^{stably} U$ in $\mathcal C([0,T];\mathbb R^2)$ and $U$ satisfies \eqref{symEulerU}.
\end{proof}

\begin{thm} \label{modifiedthmadd}
    Let $\tilde{U}_P^n(t):=n(\tilde{P}_t^n-P_t)$ and $\tilde{U}_Q^n(t):=n(\tilde{Q}_t^n-Q_t)$, where $(P,Q)$ is the solution to \eqref{Generaladdequ2d} with $d=1$ and $(\tilde{P}^n,\tilde{Q}^n)$ is the truncated strong modified equation of the symplectic Euler method given in \eqref{MEofsymadd}. 
    Then under \Cref{assumption3}, we have $(\tilde{U}_P^n,\tilde{U}_Q^n)\Rightarrow^{stably} U=(U_P,U_Q)$ in $\mathcal C ([0,T];\mathbb R^2)$ and $U$ satisfies \eqref{generalUadd} with $\theta=1$. 
\end{thm}
\begin{proof}
    Comparing \eqref{Generaladdequ2d} and \eqref{MEofsymadd}, we have $(\tilde{P}^n,\tilde{Q}^n)\to^{\mathbb{P}}(P,Q)$ as $n\to \infty$. Moreover, we obtain
    \begin{align*}
        \tilde{U}_P^n(t) = &\ n\int_0^t[f(\tilde{P}_s^n,\tilde{Q}_s^n)-f(P_s,Q_s)]ds \\
        &+n\int_0^t(f_1'f-f_2'g+\frac{1}{2}a^2f_{11}''+\frac{1}{2}b^2f_{22}'')(\tilde{P}_s^n,\tilde{Q}_s^n)(s-\eta_n(s))ds \\
        &+n\int_0^taf_1'(\tilde{P}_s^n,\tilde{Q}_s^n)(s-\eta_n(s))dW_s -n\int_0^t bf_2'(\tilde{P}_s^n,\tilde{Q}_s^n)(W_s-W_{\eta_n(s)})ds.
    \end{align*}
    For the last term of $\tilde{U}_P^n(t)$, we see that
    \begin{align*}
        n\int_0^t bf_2'(\tilde{P}_s^n,\tilde{Q}_s^n)(W_s-W_{\eta_n(s)})ds &= n\int_0^tbf_2'(\tilde{P}_{\eta_n(s)}^n,\tilde{Q}_{\eta_n(s)}^n)(W_s-W_{\eta_n(s)})ds \\
        &\quad +n\int_0^t abf_{12}''(\tilde{P}_{\eta_n(s)}^n,\tilde{Q}_{\eta_n(s)}^n)(W_s-W_{\eta_n(s)})^2ds \\
        &\quad +n\int_0^t b^2f_{22}''(\tilde{P}_{\eta_n(s)}^n,\tilde{Q}_{\eta_n(s)}^n)(W_s-W_{\eta_n(s)})^2ds + R^n(t)
    \end{align*}
    where $R^n\Rightarrow 0$ in $\mathcal{C}([0,T];\mathbb{R})$.
    In conclusion, by \eqref{nsds}-\eqref{nW^2ds} and \cite[Lemma 3.2]{KurtzProtter1991b}, we have
    \begin{align*}
        \tilde{U}_P^n(t) = \int_0^t[f_1'(P_s,Q_s)\tilde{U}_P^n+f_2'(P_s,Q_s)\tilde{U}_Q^n]ds+\tilde{M}_1^n(t),
    \end{align*}
    where
    \begin{align*}
        \tilde{M}_1^n\Rightarrow &\ \frac{1}{2}\int_0^{\cdot} (f_1'f-f_2'g)(P_s,Q_s)ds+\frac{1}{4}\int_0^{\cdot}(a^2f_{11}''-b^2f_{22}''-2abf_{12}'')(P_s,Q_s)ds \\
        &+\frac{1}{2}\int_0^{\cdot}(af_1'-bf_2')(P_s,Q_s)dW_s-\frac{\sqrt{3}}{6}\int_0^{\cdot}(af_1'+bf_2')(P_s,Q_s)dB_s.
    \end{align*}
    Similarly,
    \begin{align*}
        \tilde{U}_Q^n(t) = \int_0^t[g_1'(P_s,Q_s)\tilde{U}_P^n+g_2'(P_s,Q_s)\tilde{U}_Q^n]ds+\tilde{M}_2^n(t),
    \end{align*}
    where
    \begin{align*}
        \tilde{M}_2^n \Rightarrow &\ \frac{1}{2}\int_0^{\cdot} (g_1'f-g_2'g)(P_s,Q_s)ds+\frac{1}{4}\int_0^{\cdot}(a^2g_{11}''-b^2g_{22}''-2abg_{12}'')(P_s,Q_s)ds \\
    &+\frac{1}{2}\int_0^{\cdot}(ag_1'-bg_2')(P_s,Q_s)dW_s-\frac{\sqrt{3}}{6}\int_0^{\cdot}(ag_1'+bg_2')(P_s,Q_s)dB_s.
    \end{align*}
    Thus, by \Cref{lemmastable} and \Cref{rmk1}, we have $(\tilde{U}_P^n,\tilde{U}_Q^n)\Rightarrow^{stably} U$ in $\mathcal C([0,T]; \mathbb R^2)$ and $U$ satisfies \eqref{generalUadd} with $\theta=1$.
\end{proof}
\begin{remark}
Our results demonstrate that the truncated stochastic modified equation provides a new approach for deriving the asymptotic error distribution. This approach is more straightforward, as the truncated stochastic modified equation is formulated in terms of integrals with continuous-time adapted integrands, which facilitates the application of weak limit theorems.
In fact, denoting by $\tilde{X}^n$ the constructed truncated stochastic modified equation, by $X^n$ the continuous numerical solution, and by $X$ the exact solution, we have
\begin{align*}
    n^p(X^n-X)&=n^p(X^n-\tilde{X}^n)+n^p(\tilde{X}^n-X),
\end{align*}
where $p$ is the strong convergence order of the numerical method. 
Since $n^p(X^n-\tilde{X}^n)\Rightarrow 0$, which follows from the fact that the truncated stochastic modified equation converges to the numerical method with order higher than $p$, we can 
derive the asymptotic error distribution by calculating the limiting distribution of $n^p(\tilde{X}^n-X)$. 

\end{remark}

\section{Numerical experiments} \label{Sectionsix}
In this section, we consider the stochastic Kubo oscillator and the linear stochastic oscillator as two concrete examples, which correspond to SHSs with multiplicative and additive noise, respectively, and perform numerical experiments on the normalized Hamiltonian deviation,  demonstrating the superiority of the symplectic methods.



\subsection{Stochastic Kubo oscillator}
We consider the stochastic Kubo oscillator
\begin{align} \label{Kubo}
    d \begin{pmatrix} P_t \\ Q_t
\end{pmatrix}=\begin{pmatrix}
    -Q_t \\
    P_t
\end{pmatrix}dt +\begin{pmatrix}
    -Q_t \\ P_t
\end{pmatrix} \circ dW_t, \quad t\in (0,T],
\end{align}
with initial value $(P_0,Q_0)=(0,1)$, whose solution admits the explicit expression $P_t = -\sin(t+W_t)$ and $Q_t = \cos(t+W_t)$. The Hamiltonian function is $H(P,Q)=\frac{1}{2}P^2+\frac{1}{2}Q^2\equiv 1$.

Let $(P_{\text{E},t}^n, Q_{\text{E},t}^n)$ and $(P_{\text{sym},t}^n, Q_{\text{sym},t}^n)$ denote the continuous numerical solutions of the Euler and symplectic methods \eqref{generalsym} for \eqref{Kubo}, respectively. Their corresponding asymptotic error distributions are denoted by $U_{\text{E}}:=(U_{\text{E},P}, U_{\text{E},Q})$ and $U_{\text{sym}}:=(U_{\text{sym},P}, U_{\text{sym},Q})$.
It follows from \cite[Theorem 3.2]{JacodProtter1998} and \Cref{mainthmmulti} that 
\begin{align*}
    &\begin{cases} 
        dU_{\text{E},P}(t)=(-\frac{1}{2}U_{\text{E},P}(t)-U_{\text{E},Q}(t))\,dt-U_{\text{E},Q}(t)\,dW_t+\frac{1}{\sqrt{2}}P_t\,dB_t, \quad U_{\text{E},P}(0)=0, \\
        dU_{\text{E},Q}(t)=(U_{\text{E},P}(t)-\frac{1}{2}U_{\text{E},Q}(t))\,dt+U_{\text{E},P}(t)\,dW_t+\frac{1}{\sqrt{2}}Q_t\,dB_t, \quad U_{\text{E},Q}(0)=0, \\
    \end{cases} \\ 
    &\begin{cases} 
        dU_{\text{sym},P}(t)=(-\frac{1}{2}U_{\text{sym},P}(t)-U_{\text{sym},Q}(t))\,dt-U_{\text{sym},Q}(t)\,dW_t+\frac{2\theta-1}{\sqrt{2}}P_t\,dB_t, \quad U_{\text{sym},P}(0)=0, \\
        dU_{\text{sym},Q}(t)=(U_{\text{sym},P}(t)-\frac{1}{2}U_{\text{sym},Q}(t))\,dt+U_{\text{sym},P}(t)\,dW_t-\frac{2\theta-1}{\sqrt{2}}Q_t\,dB_t, \quad U_{\text{sym},Q}(0)=0.
    \end{cases}  
\end{align*}
Solving the above equations gives $U_{\text{E},P}(t) = -\frac{1}{\sqrt{2}}B_t\sin(t+W_t)$, $U_{\text{E},Q}(t) = \frac{1}{\sqrt{2}} B_t\cos(t+W_t)$, and
\begin{align*}
    \begin{cases}
        U_{\text{sym},P}(t) = -\frac{2\theta-1}{\sqrt{2}}\cos(t+W_t)\int_0^t \sin(2s+2W_s)dB_s +\frac{2\theta-1}{\sqrt{2}}\sin(t+W_t)\int_0^t\cos(2s+2W_s)dB_s, \\
        U_{\text{sym},Q}(t) = -\frac{2\theta-1}{\sqrt{2}}\sin(t+W_t)\int_0^t \sin(2s+2W_s)dB_s -\frac{2\theta-1}{\sqrt{2}}\cos(t+W_t)\int_0^t\cos(2s+2W_s)dB_s.
    \end{cases}
\end{align*}

With these preparations, we give the specific expression for the limiting distribution of the normalized Hamiltonian deviation. 
In fact, applying Theorem \ref{thmdeltaH} yields $\sqrt{n}(H(P_t^n,Q_t^n)-H(P_t,Q_t)) \Rightarrow^{stably} P_tU_P(t)+Q_tU_Q(t)$ in $\mathcal C([0,T];\mathbb R)$. 
Combining this with 
\begin{align*}
    &P_tU_{\text{E},P}(t)+Q_tU_{\text{E},Q}(t) = \frac{1}{\sqrt{2}}B_t,\\
    &P_tU_{\text{sym},P}(t)+Q_tU_{\text{sym},Q}(t) = -\frac{2\theta-1}{\sqrt{2}}\int_0^t \cos(2s+2W_s)dB_s,
\end{align*}
we theoretically obtain 
\begin{align*}
    &\lim_{n\to\infty} \sqrt{n}\,\mathbb E[H(P_{\text{E},t}^n,Q_{\text{E},t}^n)-H(P_t,Q_t)] =0,\\
    &\lim_{n\to\infty}\sqrt{n}\,\mathbb E[H(P_{\text{sym},t}^n,Q_{\text{sym},t}^n)-H(P_t,Q_t)] = 0,
\end{align*}
which is consistent with the property that the weak convergence orders of both the Euler method and the symplectic methods are 1. Furthermore, we also have
\begin{align*}
    &\lim_{n\to\infty} n\mathbb E\big[(H(P_{\text{E},t}^n,Q_{\text{E},t}^n)-H(P_t,Q_t))^2\big] =\frac{1}{2}t, \\
    &\lim_{n\to\infty} n \mathbb E\big[(H(P_{\text{sym},t}^n,Q_{\text{sym},t}^n)-H(P_t,Q_t))^2\big] = \frac{(2\theta-1)^2}{2}\int_0^t\tilde{\mathbb E}\big[(\cos(2s+2W_s))^2\big]ds\le \frac{1}{2}t.
\end{align*}

Below, we present numerical experiments to verify these results. The Monte Carlo method is employed to estimate the expectation in all experiments. 
In \Cref{fig1:total}, we use 10000 sample paths with final time $T=4$ and discretization parameters $n=2,5,10,20,25,50,100,200$ in \Cref{fig1:total}(A) and $n=5,10,20,25,50,100$ in \Cref{fig1:total}(B), illustrating that the symplectic Euler method behaves better than the Euler method even after taking limit for $n$.
\Cref{fig2:total} presents results by 50,000 sample paths over the time interval $[0,10]$, computing $n\mathbb E[(H_t^n-H_t)^2]$ with $n=100$ in \Cref{fig2:total}(A) and $n=1000$ in \Cref{fig2:total}(B).
We observe that, compared with the Euler method, the symplectic methods indeed exhibit a noticeably slower growth in $n\mathbb E[(H_t^n-H_t)^2]$ over time. Furthermore, our results show that as $\theta$ approaches $0.5$, the limit tends to $0$, which is consistent with our theoretical result.
\begin{figure}[htbp]
  \centering
  \begin{subfigure}[b]{0.4\textwidth} 
    \centering
    \includegraphics[width=\textwidth]{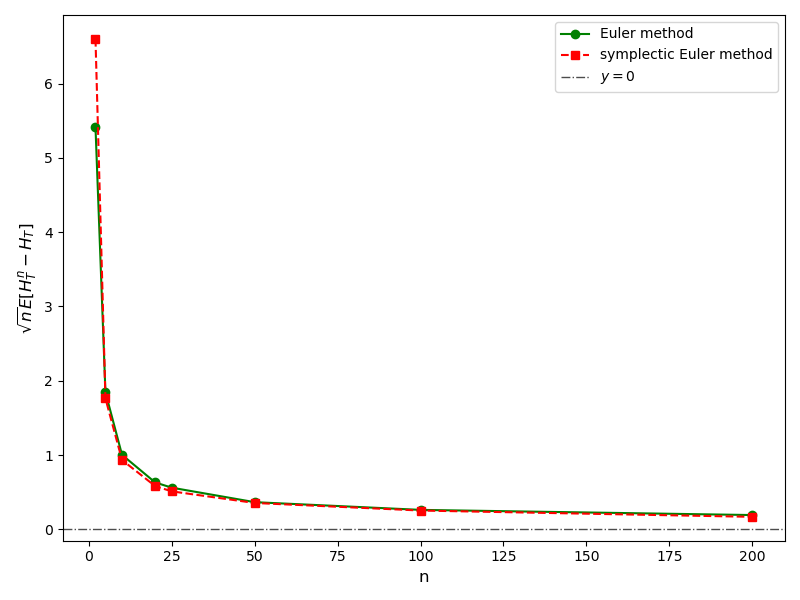} 
    \caption{$T=4$}
    \label{fig1:sub1}
  \end{subfigure}
  \hspace{2mm}
  \begin{subfigure}[b]{0.4\textwidth} 
    \centering
    \includegraphics[width=\textwidth]{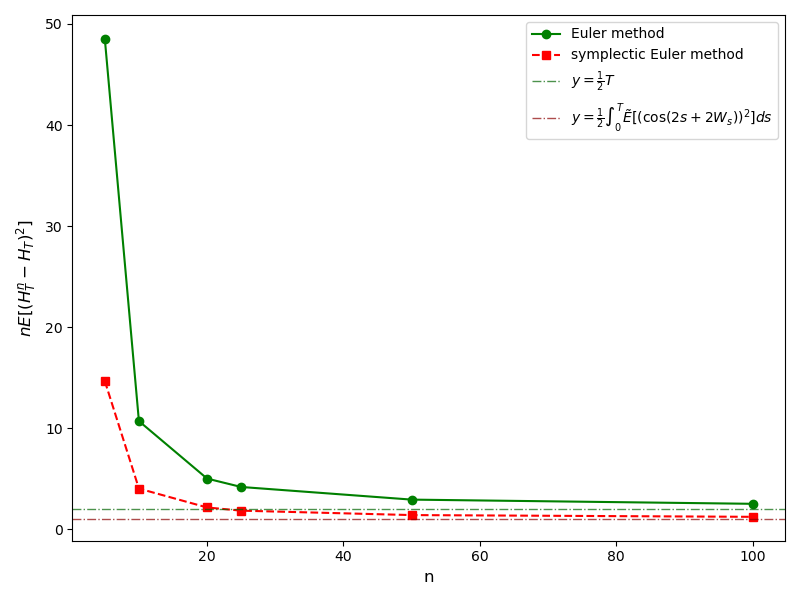}
    \caption{$T=4$}
    \label{fig1:sub2}
  \end{subfigure}
  \caption{$\sqrt{n}\mathbb E[H_T^n-H_T]$ and $n\mathbb E[(H_T^n-H_T)^2]$ for fixed $T$.}
  \label{fig1:total}
\end{figure}
\begin{figure}[htbp]
  \centering
  \begin{subfigure}[]{0.4\textwidth} 
    \centering
    \includegraphics[width=\textwidth]{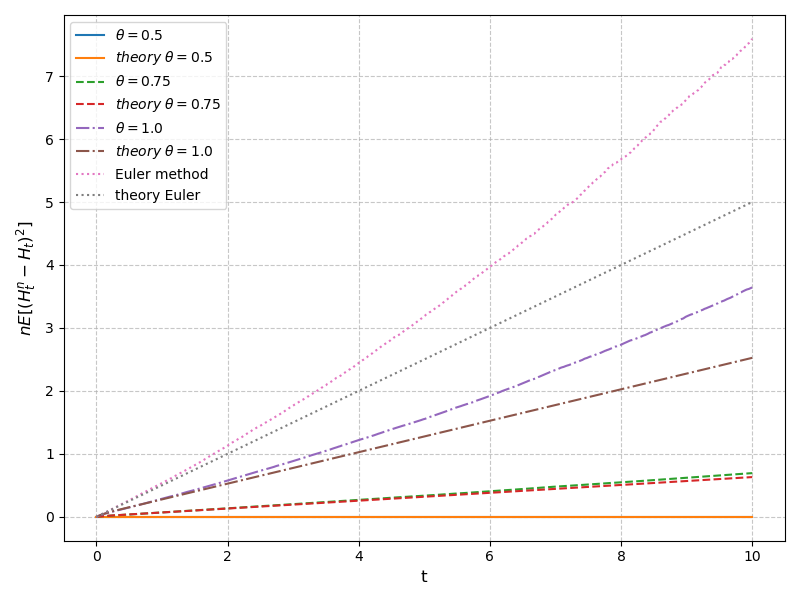} 
    \caption{$n=100$}
    \label{fig2:sub1}
  \end{subfigure}
  \hspace{2mm}
  \begin{subfigure}[]{0.4\textwidth} 
    \centering
    \includegraphics[width=\textwidth]{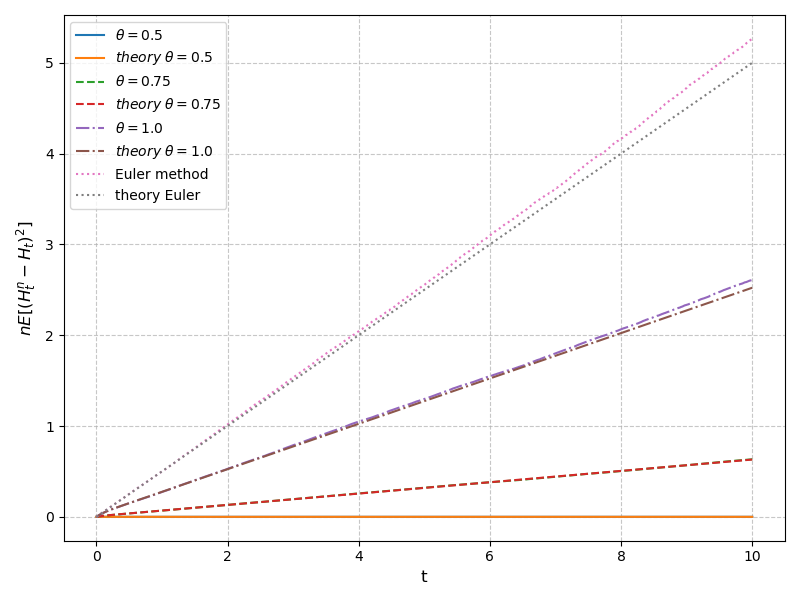}
    \caption{$n=1000$}
    \label{fig2:sub2}
  \end{subfigure}
  \caption{$n\mathbb E[(H_t^n-H_t)^2]$ for fixed $n$.}
  \label{fig2:total}
\end{figure}

\subsection{Linear stochastic oscillator}
Before presenting our results, we first give the asymptotic error distribution $U_{\text{E}}:=(U_{\text{E},P},U_{\text{E},Q})$ of Euler method $(P_{\text{E},t}^n,Q_{\text{E},t}^n)$ for \eqref{Generaladdequ2d} with $d=1$,
\begin{align*}
        U_{\text{E},P}(t) =& \int_0^t[f_1'(P_s,Q_s)U_{\text{E},P}(s)+f_2'(P_s,Q_s)U_{\text{E},Q}(s)]ds \\
        &-\frac{1}{2}\int_0^t (f_1'f+f_2'g)(P_s,Q_s)ds-\frac{1}{4}\int_0^t(a^2f_{11}''+b^2f_{22}''+2abf_{12}'')(P_s,Q_s)ds \\
        &-\frac{1}{2}\int_0^t(af_1'+bf_2')(P_s,Q_s)dW_s-\frac{\sqrt{3}}{6}\int_0^t(af_1'+bf_2')(P_s,Q_s)dB_s,\\
        U_{\text{E},Q}(t) =& \int_0^t[g_1'(P_s,Q_s)U_{\text{E},P}(s)+g_2'(P_s,Q_s)U_{\text{E},Q}(s)]ds \\
        &-\frac{1}{2}\int_0^t (g_1'f+g_2'g)(P_s,Q_s)ds-\frac{1}{4}\int_0^t(a^2g_{11}''+b^2g_{22}''+2abg_{12}'')(P_s,Q_s)ds \\
        &-\frac{1}{2}\int_0^t(ag_1'+bg_2')(P_s,Q_s)dW_s-\frac{\sqrt{3}}{6}\int_0^t(ag_1'+bg_2')(P_s,Q_s)dB_s.
\end{align*}
Then, in this subsection, we consider the linear stochastic oscillator
\begin{align} \label{linear}
    d \begin{pmatrix} P_t \\ Q_t
\end{pmatrix}=\begin{pmatrix}
    -Q_t \\
    P_t
\end{pmatrix}dt +\begin{pmatrix}
    1 \\ 0
\end{pmatrix} dW_t, \quad t\in (0,T],
\end{align}
with initial value $(P_0,Q_0)=(0,0)$, whose solution admits the explicit expression $P_t = \int_0^t \cos(t-s)dW_s$ and $Q_t = \int_0^t \sin(t-s)dW_s$. 
In this case, the asymptotic error distributions $U_{\text{E}}$ and $U_{\text{sym}}$ satisfy
\begin{align*}
    &\begin{cases} 
        dU_{\text{E},P}(t)=-U_{\text{E},Q}(t)dt+\frac{1}{2}P_tdt, \quad U_{\text{E},P}(0)=0, \\
        dU_{\text{E},Q}(t)=U_{\text{E},P}(t)dt+\frac{1}{2}Q_tdt-\frac{1}{2}dW_t-\frac{\sqrt{3}}{6}dB_t, \quad U_{\text{E},Q}(0)=0, \\
    \end{cases} \\[1ex]
    &\begin{cases} 
        dU_{\text{sym},P}(t)=-U_{\text{sym},Q}(t)dt+(\theta-\frac{1}{2})P_tdt, \quad U_{\text{sym},P}(0)=0, \\
        dU_{\text{sym},Q}(t)=U_{\text{sym},P}(t)dt-(\theta-\frac{1}{2})Q_tdt+(\theta-\frac{1}{2})dW_t-\frac{\sqrt{3}}{6}dB_t, \quad U_{\text{sym},Q}(0)=0.
    \end{cases}  
\end{align*}

With these preparations, we proceed to analyze the normalized Hamiltonian deviation corresponding to the Hamiltonian function $H(P,Q)=\frac{1}{2}P^2+\frac{1}{2}Q^2$ of the original equation \eqref{linear}. By Theorem \ref{thmdeltaHadd}, we have $n(H(P_t^n,Q_t^n)-H(P_t,Q_t)) \Rightarrow^{stably} P_tU_P(t)+Q_tU_Q(t)$ in $\mathcal{C}([0,T],\mathbb{R})$. It follows from the It\^o formula that
\begin{align*}
    &d(P_tU_{\text{E},P}(t)+Q_tU_{\text{E},Q}(t))=\frac{1}{2}(P_t^2+Q_t^2)dt+U_{\text{E},P}(t)dW_t-\frac{1}{2}Q_tdW_t-\frac{\sqrt{3}}{6}Q_tdB_t, \\
    &d(P_tU_{\text{sym},P}(t)+Q_tU_{\text{sym},Q}(t))=(\theta-\frac{1}{2})(P_t^2-Q_t^2)dt+U_{\text{sym},P}(t)dW_t+(\theta-\frac{1}{2})Q_tdW_t-\frac{\sqrt{3}}{6}Q_tdB_t.
\end{align*}
Thus, we theoretically obtain
\begin{align*}
    &\lim_{n\to\infty}n\mathbb E[H(P_{\text{E},t}^n,Q_{\text{E},t}^n)-H(P_t,Q_t)] = \frac{1}{2}\int_0^t\tilde{\mathbb E}[(P_s^2+Q_s^2)]ds=\frac{1}{4}t^2,\\
    &\lim_{n\to\infty}n\mathbb E[H(P_{\text{sym},t}^n,Q_{\text{sym},t}^n)-H(P_t,Q_t)] =(\theta- \frac{1}{2})\int_0^t\tilde{\mathbb E}[(P_s^2-Q_s^2)]ds=\frac{1}{4}(\theta-\frac{1}{2})(1-\cos(2t)).
\end{align*}

We then present numerical experiments to verify these results. In \Cref{fig3:total}, we use $2\times10^8$ sample paths with final time $T=4$ and discretization parameters $n=5,10,20,25,40,50,100$ in \Cref{fig3:total}(A) and $n=2,5,10,20,25,40,50,100$ in \Cref{fig3:total}(B).
The simulations for \Cref{fig4:total} employ $2\times10^5$ sample paths on $[0,20]$ with $n=100$ for \Cref{fig4:total}(A), and $2\times 10^8$ sample paths on $[0,6]$ with $n=40$ for \Cref{fig4:total}(B)-(D), with $\theta=1,0.75,0.1$ for (B), (C) and (D), respectively.
\begin{figure}[htbp] 
  \centering
  \begin{subfigure}[b]{0.4\textwidth} 
    \centering
    \includegraphics[width=\textwidth]{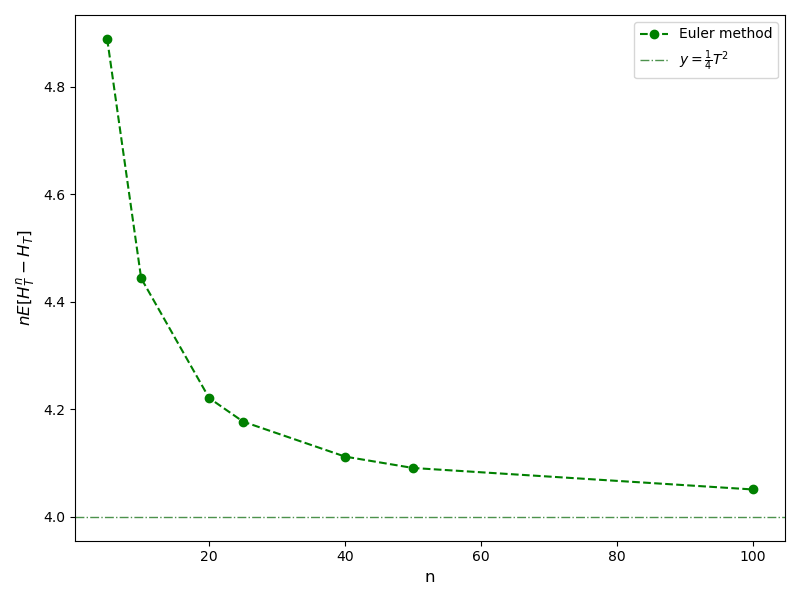} 
    \caption{$T=4$}
    \label{fig3:sub1}
  \end{subfigure}
  \hspace{2mm}
  \begin{subfigure}[b]{0.4\textwidth} 
    \centering
    \includegraphics[width=\textwidth]{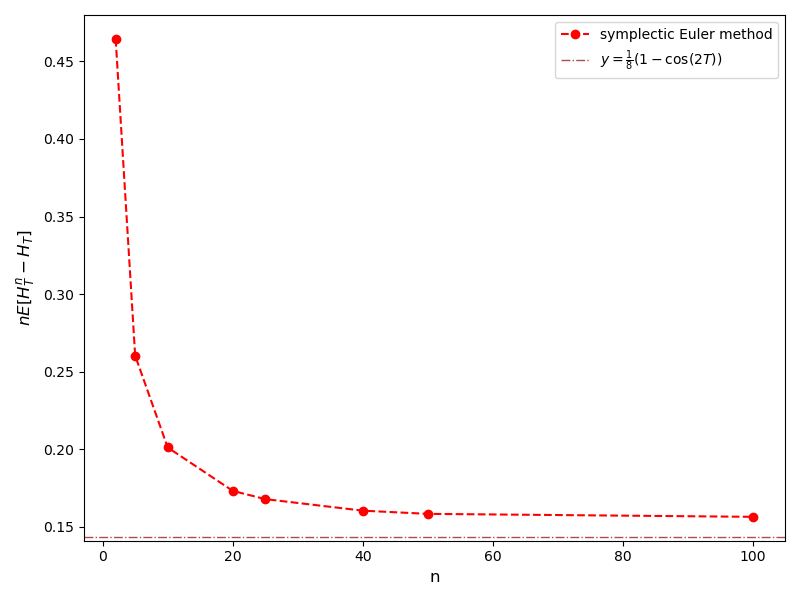}
    \caption{$T=4$}
    \label{fig3:sub2}
  \end{subfigure}
  \caption{$n\mathbb E[H_T^n-H_T]$ for fixed $T$.}
  \label{fig3:total}
\end{figure}
\begin{figure}[htbp] 
  \centering
  \begin{subfigure}[b]{0.36\textwidth} 
    \centering
    \includegraphics[width=\textwidth]{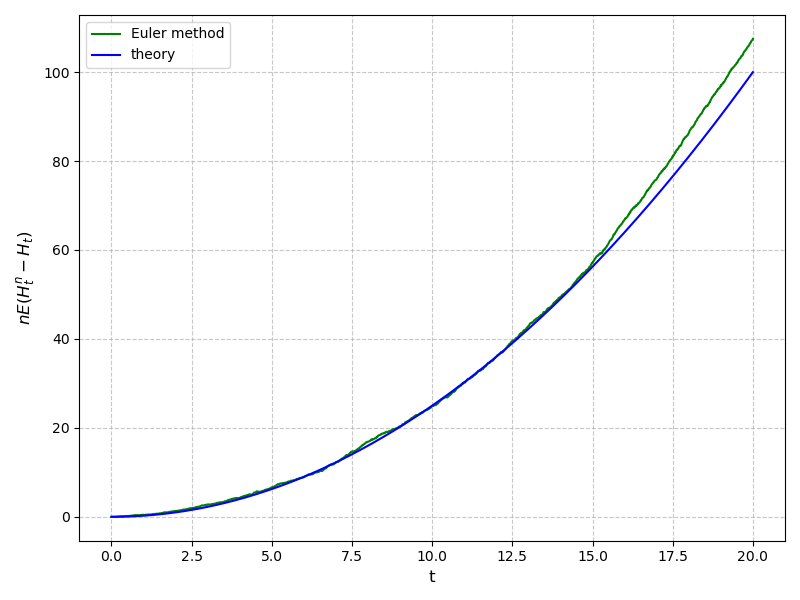} 
    \caption{$n=100$}
    \label{fig4:sub1}
  \end{subfigure}
  \hspace{2mm}
  \begin{subfigure}[b]{0.36\textwidth} 
    \centering
    \includegraphics[width=\textwidth]{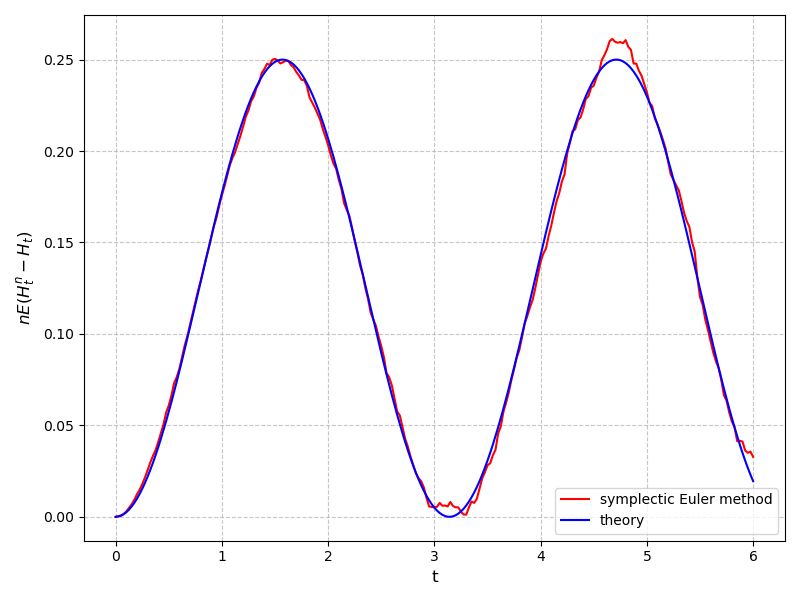}
    \caption{$n=40$}
    \label{fig4:sub2}
  \end{subfigure}
  \begin{subfigure}[b]{0.36\textwidth} 
    \centering
    \includegraphics[width=\textwidth]{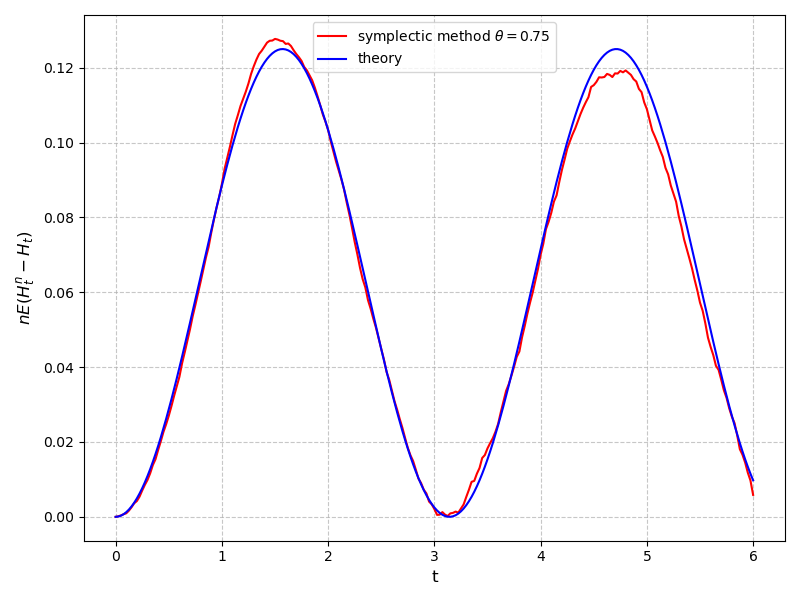}
    \caption{$n=40$}
    \label{fig4:sub3}
  \end{subfigure}
  \hspace{2mm}
  \begin{subfigure}[b]{0.36\textwidth} 
    \centering
    \includegraphics[width=\textwidth]{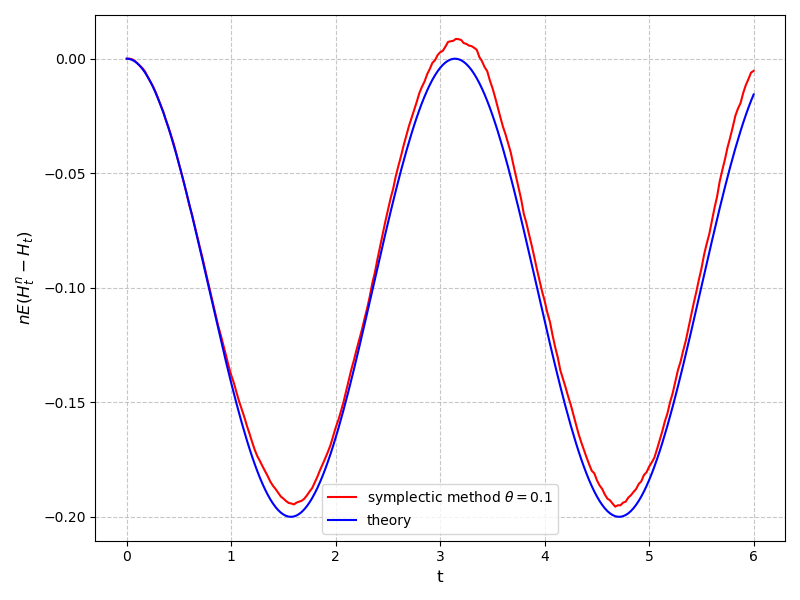}
    \caption{$n=40$}
    \label{fig4:sub4}
  \end{subfigure}
  \caption{$n\mathbb E[H_t^n-H_t]$ for fixed $n$.}
  \label{fig4:total}
\end{figure}

Our numerical results verify that symplectic methods can better simulate the original Hamiltonians even in the limit as $n$ tends to infinity. Specifically, in the limit as $n$ approaches infinity, the normalized Hamiltonian deviation grows quadratically with time for the Euler method, whereas for symplectic methods it oscillates and remains bounded.

	\bibliographystyle{plain}
	\bibliography{references}
		
\end{document}